\documentclass[12pt,reqno]{amsart}

\usepackage{multirow}
\usepackage{hhline}

\usepackage{mathtools}
\usepackage{times}
\usepackage[T1]{fontenc}
\usepackage{mathrsfs}
\usepackage{latexsym}
\usepackage[dvips]{graphics}
\usepackage[titletoc, title]{appendix}
\usepackage{epsfig}
\usepackage{amssymb}
\usepackage{amsmath,amsfonts,amsthm,amssymb,amscd}
\usepackage{color}
\usepackage{hyperref}
\usepackage{amsmath}

\usepackage{color}
\usepackage{breakurl}

\usepackage{comment}
\newcommand{\bburl}[1]{\textcolor{blue}{\url{#1}}}

\newcommand{\be}{\begin{equation}}
\newcommand{\ee}{\end{equation}}
\newcommand{\seqnum}[1]{\href{https://urldefense.com/v3/__https://oeis.org/*1*7D*7B*5Crm__;IyUlJQ!!DZ3fjg!804FcbFfaB3k7d40hp-AJLCHDSIqAeZVJt6te_Tkl5mMrDdnUUMrPSS_1RNJmZAfMjMrd-8s2yXSiL6QzSuPkhz2x34$  \underline{#1}}}

\DeclareMathOperator{\sgn}{sgn}
\newtheorem{thm}{Theorem}[section]

\newtheorem{cor}[thm]{Corollary}

\newtheorem{lem}[thm]{Lemma}
\newtheorem{prop}[thm]{Proposition}
\newtheorem{exa}[thm]{Example}

\newtheorem{defi}[thm]{Definition}
\newtheorem{rek}[thm]{Remark}

\newcommand{\N}{\mathbb{N}}

\DeclareMathOperator{\supp}{supp}
\numberwithin{equation}{section}

\begin{document}

\title{On Approximation Spaces and Greedy-type Bases}

\author{Pablo M. Bern\'a}

\email{\textcolor{blue}{\href{mailto:pablo.berna@cunef.edu}{pablo.berna@cunef.edu}}}
\address{Departmento de Métodos Cuantitativos, CUNEF Universidad, 28040 Madrid, Spain}

\author{H\`ung Vi\d{\^e}t Chu}

\email{\textcolor{blue}{\href{mailto:hungchu2@illinois.edu}{hungchu2@illinois.edu}}}
\address{Department of Mathematics, University of Illinois at Urbana-Champaign, Urbana, IL 61820, USA}

\author{Eugenio Hern\'andez}

\email{\textcolor{blue}{\href{mailto:eugenio.hernandez@uam.es}{eugenio.hernandez@uam.es}}}
\address{	Departmento de Matem\'aticas, Universidad Aut\'onoma de Madrid, 28049 Madrid, Spain}

\subjclass[2020]{}


\thanks{The first and third authors were supported by Grants PID2019-105599GB-I00 (Agencia Estatal de Investigación, Spain). This work has been supported by the Madrid Government (Spain) under the multiannual Agreement with UAM in the line for the Excellence of the University Research Staff in the context of the V PRICIT (Regional Programme of Research and Technological Innovation).}

\maketitle

\begin{abstract}
The purpose of this paper is to introduce $\omega$-Chebyshev-greedy and $\omega$-partially greedy approximation classes and to study their  relation with $\omega$-approximation spaces, where the latter are a generalization of the classical approximation spaces. The relation gives us sufficient conditions of when certain continuous embeddings imply different greedy-type properties. Along the way, we generalize a result by P. Wojtaszczyk as well as characterize semi-greedy Schauder bases in quasi-Banach spaces, generalizing a previous result by the first author.
\end{abstract}

\tableofcontents

\section{Background and main results}
Let $(\mathbb{X}, \|\cdot\|)$ be a quasi-Banach space; that is, $\mathbb{X}$ is a complete vector space over $\mathbb F=$ $\mathbb{R}$ or $\mathbb C$, with a quasi-norm $\|\cdot\|: \mathbb{X}\rightarrow [0,\infty)$ such that
\begin{itemize}
	\item[(C1)] $\|x\| = 0$ if and only if $x = 0$, 
	
	\item[(C2)] $\| t x\|=|t| \| x\|$ for all $t\in\mathbb{F}$ and all $x\in \mathbb{X}$, and
	
	\item[(C3)] there is a constant $\kappa\ge 1$ so that 
	\begin{equation*}\| x +y\| \ \le\ \kappa(\| x\| +\| y\|),\forall x, y\in \mathbb{X}.\end{equation*}
\end{itemize}

Given $0<p\le 1$, a $p$-norm is a map $\| \cdot\|: \mathbb{X}\rightarrow [0, \infty)$ satisfying (C1), (C2), and
\begin{itemize}
	\item[(C4)] $\| x+y\|^{p}\le \| x\|^{p} +\| y\|^p$
	for all $x,y\in \mathbb{X}$.
\end{itemize}
It follows that the condition $(C4)$ implies $(C3)$ with $\kappa=2^{1/p-1}$. A quasi-Banach space whose quasi-norm is a $p$-norm shall be called a $p$-Banach space. Thanks to the Aoki-Rolewicz's Theorem (see \cite{Aoki,Rolewicz}), we know that any quasi-Banach space $\mathbb{X}$ is $p$-convex for some $0<p\le 1$, i.e., there is a constant $C$  such that
$$\left\| \sum_{j=1}^n x_j\right\|^p \le C  \sum_{j=1}^n \| x_j\|^p, \forall n\in\mathbb{N}, \forall  x_j\in\mathbb{X}.$$
This way, a quasi-Banach space becomes $p$-Banach under a suitable renorming. 

We say that $\mathcal B=(e_n)_{n=1}^\infty$ is a semi-normalized Markushevich basis (or an M-basis or simply a basis) of $\mathbb{X}$ if the following holds
\begin{itemize}
    \item[a1)] There exists a (unique) collection $(e_n^*)_{n=1}^\infty\subset (\mathbb{X}^*, \|\cdot\|_*)$, called the biorthogonal functionals, such that $e_i^*(e_j)=\delta_{i,j}$.
	\item[a2)] There exist $c_1, c_2 > 0$ such that $0< c_1 \le \{\|e_n\|,\|e_n^*\|_*\} \le c_2 <\infty$ for all $n\in\N$.
	\item[a3)] $\mathbb X=\overline{\mbox{span}\{ e_n : n\in\N\}}$.
	\item[a4)] If $e_n^*(x)=0$ for all $n\in\N$, then $x=0$.
\end{itemize}

Under these conditions, every $x\in\mathbb{X}$ is represented by $x\sim\sum_{n=1}^\infty e_n^*(x)e_n$ and $\lim_{n\rightarrow\infty}e_n^*(x)=0$. Also, if we consider the algorithm of partial sums $(S_m)_{m=1}^\infty$ defined as $S_m(x)=\sum_{n=1}^m e_n^*(x)e_n$, we say that $\mathcal B$ is a  Schauder basis if there is a positive constant $C$ such that
\begin{equation}\label{sch}
\| S_m(x)\| \ \le\ C\| x\|, \forall m\in\mathbb{N}, \forall x\in\mathbb{X}.
\end{equation}
We denote by $K_b$ the smallest constant in \eqref{sch}, which is called the basis constant.

 One of the main objects studied in Approximation Theory is the nonlinear approximation spaces  $\mathcal A_q^\alpha(\mathbb{X},\mathcal B)$: for $\alpha>0$ and $0<q < \infty$,
$$\mathcal A_q^\alpha(\mathcal B, \mathbb{X})\ =\ \mathcal A_q^\alpha \ :=\ \left\{ x\in\mathbb{X} : \| x\|_{\mathcal A_q^\alpha}:= \| x\| + \left[\sum_{n=1}^\infty (n^\alpha \sigma_n(x))^{q}\frac{1}{n}\right]^{1/q}<\infty\right\}$$
and 
$$\mathcal A_\infty^\alpha(\mathcal B, \mathbb{X})\ =\ \mathcal A_\infty^\alpha \ :=\ \left\{ x\in\mathbb{X} : \| x\|_{\mathcal A_q^\alpha}:= \| x\| + \sup_{n\ge 1}n^\alpha \sigma_n(x)<\infty\right\},$$
where $\sigma_m(x)$ is the $m$-term error of approximation:
\begin{equation}\label{ee1}\sigma_m(\mathcal B, x)_{\mathbb{X}}\ =\ \sigma_m(x)\ :=\ \inf\left\{\left\| x-\sum_{n\in A}b_ne_n\right\| : | A|= m, b_n \in\mathbb{F}\right\}.\end{equation}

There have been studies on embeddings among these approximation spaces and Lorentz spaces (see \cite{DT,GN,KN,S}). One of the recent results is given in \cite{GHN}:

\begin{thm}{\cite[Theorem 1.4]{GHN}}
Let $\mathbb{X}$ be a quasi-Banach space with an unconditional M-basis $\mathcal B$. Assume that $h_l(m)$ is a doubling function. Then for $\alpha>0$ and $q\in (0,\infty]$, we have the following continuous embeddings
$$\ell_{k^\alpha h_r(k)}^q(\mathcal B, \mathbb{X})\ \hookrightarrow\ \mathcal A_q^\alpha(\mathcal B,\mathbb{X})\ \hookrightarrow\ \ell_{k^\alpha h_l(k)}^q(\mathcal B, \mathbb{X})\footnote{Throughout this paper, $X\hookrightarrow Y$ means $X\subset Y$ and there exists $C>0$ such that $\|x\|_Y\le C\|x\|_X$ for all $x\in X$. Moreover, $X\approx Y$ means $X\hookrightarrow Y$ and $Y\hookrightarrow X$.},$$
where $h_l(m)$ and $h_r(m)$ are the so-called democracy functions.
\end{thm}

Also in \cite{GHN}, the authors introduced a new class, called the greedy class, using the Thresholding Greedy Algorithm (TGA) \cite{KT}: a greedy sum of $x$ of order $m$ is given by 
$$G^{\pi}_m(x)\ =\ \sum_{n=1}^me_{\pi(n)}^{*}(x)e_{\pi(n)},$$
where $\pi$ is a greedy ordering, i.e., $\pi: \mathbb{N}\longrightarrow\mathbb{N}$  is a permutation such that $\supp(x)\subset \pi(\mathbb{N})$ and $| e^*_{\pi(i)}(x)| \ge | e^*_{\pi(j)}(x)|$ whenever $i\le j$. Also, $A=\{\pi(n): 1\le n\le m\}$ is called a greedy set of $x$ of order $m$. The $m$th greedy error for $x\in\mathbb{X}$ is the quantity
\begin{equation}\label{gerror}
\gamma_m(x)\ :=\ \sup_{\pi}\| x-G_m^\pi(x)\|.
\end{equation}
The greedy class $\mathcal G_q^\alpha$ is defined as: for $\alpha>0$ and $q\in (0,\infty)$,
$$\mathcal G_q^\alpha(\mathcal B, \mathbb{X})=\mathcal G_{q}^\alpha \ :=\ \left\{ x\in\mathbb{X} : \| x\|_{\mathcal G_q^\alpha}\ :=\ \| x\| + \left[\sum_{n=1}^\infty (n^\alpha \gamma_n(x))^{q}\frac{1}{n}\right]^{1/q}<\infty\right\}.$$
The class $\mathcal G_\infty^\alpha$ is defined in the same way as $\mathcal A_\infty^\alpha$ with $\sigma_n$  replaced by $\gamma_n$.
One of the results in \cite{GHN} is that if an M-basis $\mathcal B$ in a quasi-Banach space is greedy, then $\mathcal A_q^\alpha \approx \mathcal G_q^\alpha$. The converse (in the context of Banach spaces) was given by Wojtaszczyk \cite{W}:

\begin{thm}\cite[Theorem 3.1 (restated)]{W}\label{woj} Let $\mathcal{B}$ be an unconditional M-basis in a Banach space $\mathbb{X}$.
\begin{enumerate}
    \item If $\mathcal B$ is greedy, then $\mathcal A_q^\alpha\approx \mathcal G_q^\alpha$ for all $q\in (0,\infty]$ and $\alpha > 0$.
    \item If $\mathcal A_q^\alpha\approx \mathcal G_q^\alpha$ for some $q\in (0,\infty]$ and $\alpha > 0$, then $\mathcal B$ is greedy. 
\end{enumerate}
\end{thm}

Statement (1) in Theorem \ref{woj} follows from definitions, while proving statement (2) is considerably more involved.  
The purpose of this paper is to study the above result in the more general context of weights and for different greedy-type bases. In particular, given $q\in (0,\infty]$ and a weight $\omega = (\omega(n))_{n=1}^\infty$ satisfying certain conditions (see definitions in Subsection \ref{weights}), we define the following $\omega$-approximation spaces: for $0<q<\infty$,
$$\mathcal A_q^\omega\ :=\ \left\{ x\in\mathbb{X} : \| x\|_{\mathcal A_q^\omega}:= \| x\| + \left[\sum_{n=1}^\infty (\omega(n)\sigma_n(x))^{q}\frac{1}{n}\right]^{1/q}<\infty\right\};$$
for $q=\infty$,
$$\mathcal A_\infty^\omega\ := \ \left\{ x\in\mathbb{X} : \| x\|_{\mathcal A_\infty^\omega}:=\| x\| + \sup_{n\ge 1}\omega(n)\sigma_n(x)<\infty\right\}.$$
These spaces were recently considered in \cite{CDK}. When $\omega(n)=n^\alpha$, we recover the classical approximation space $\mathcal A_q^\alpha$. We shall introduce three different greedy approximation classes. The first one is a generalization of $\mathcal G_q^\alpha$: if $0<q<\infty$ and $\omega$ is a weight,

$$\mathcal{G}_{q}^\omega\ :=\ \left\{ x\in\mathbb{X} : \| x\|_{\mathcal{G}_q^\omega}:= \| x\| + \left[\sum_{n=1}^\infty (\omega(n)\gamma_n(x))^{q}\frac{1}{n}\right]^{1/q}<\infty\right\},$$
and for $q=\infty$,
$$\mathcal G_\infty^\omega\ :=\  \left\{ x\in\mathbb{X} : \| x\|_{\mathcal G_\infty^\omega}:=\| x\| + \sup_{n\ge 1}\omega(n)\gamma_n(x)<\infty\right\}.$$
For $\omega(n)=n^\alpha$, we recover the original greedy class introduced in \cite{GHN}.

In \cite{DKKT}, Dilworth, Kalton, Kutzarova, and Temlyakov introduced a new algorithm that is an enhancement of the rate of convergence of the greedy algorithm. For $x\in\mathbb{X}$, the Thresholding Chebyshev Greedy Algorithm (TCGA) $({ CG}_m^\pi)_{m=1}^\infty$ is defined as:
$$\| x-{CG}_m^\pi(x)\|\ =\ \inf_{(a_n)\subset \mathbb{F}}\left\| x-\sum_{n=1}^m a_ne_{\pi(n)}\right\|.$$
The $m$th Chebyshev-greedy error for $x\in\mathbb{X}$ is
$$\vartheta_m(x)\ :=\ \sup_\pi\| x- {CG}_m^\pi(x)\|.$$
Using this error, we define the Chebyshev-greedy approximation class: for $0<q<\infty$ and a weight $\omega$,
$$\mathcal{CG}_{q}^\omega \ :=\ \left\{ x\in\mathbb{X} : \| x\|_{\mathcal{CG}_q^\omega}:= \| x\| + \left[\sum_{n=1}^\infty (\omega(n)\vartheta_n(x))^{q}\frac{1}{n}\right]^{1/q}<\infty\right\},$$
and for $q=\infty$, we consider the usual modification as in $\mathcal G_\infty^\omega$. Obviously, we have the following continuous embeddings: for any $q>0$ and a weight $\omega$,
$$\mathcal G_q^\omega\ \hookrightarrow\ \mathcal{CG}_q^\omega\ \hookrightarrow\ \mathcal A_q^\omega.$$

Finally, we introduce the partially greedy class $\mathcal{PG}_q^\infty$: for $q\in (0,\infty)$ and a weight $\omega$, 
$$\mathcal{PG}_q^\omega\ =\ \left\{x\in \mathbb{X}: \|x\|_{\mathcal{PG}_q^\omega} := \|x\| + \left(\sum_{n=1}^\infty (w(n)\beta_n(x))^q\frac{1}{n}\right)^{1/q}<\infty\right\},$$
where 
$\beta_m(x) \ :=\ \left\|x-S_m(x)\right\|$.

For two functions $f(a_1, a_2, \ldots)$ and $g(a_1, a_2, \ldots)$, we write $f\lesssim g$ to indicate that there exists an absolute constant $C>0$ (independent of $a_1, a_2, \ldots$) such that $f\le Cg$. Similarly, $f\gtrsim g$ means that $Cf\ge g$ for some constant $C$. Furthermore, $f \asymp g$ means that $f\lesssim g$ and $f\gtrsim g$. For two sets $A, B\subset\mathbb{N}$, we write $A<B$ to mean that $\max A < \min B$. We are ready to state the three main results of this paper.

\begin{thm}\label{main1} Let $\omega$ be a doubling weight with $i_\omega>0$. Let $\mathcal{B}$ be a quasi-greedy M-basis in a quasi-Banach space $\mathbb{X}$. 
    \begin{enumerate}
    \item If $\mathcal B$ is semi-greedy, then $\mathcal A_q^\omega \approx \mathcal{CG}_q^\omega$ for all $q\in (0,\infty]$.
    \item If $\mathcal B$ is Schauder with Property (W) and $\mathcal A_q^\omega \approx \mathcal{CG}_q^\omega$ for some $q\in (0,\infty]$, then $\mathcal B$ is semi-greedy. 
\end{enumerate}
\end{thm}

The techniques used to prove the Theorem \ref{main1} allow us to generalize Theorem \ref{woj}.

\begin{thm}\label{ghn}
	Let $\omega$ be a doubling weight with $i_\omega>0$. Let $\mathcal{B}$ be an unconditional M-basis in a quasi-Banach space $\mathbb{X}$.
	\begin{enumerate}
    \item If $\mathcal B$ is greedy, then $\mathcal A_q^\omega\approx \mathcal G_q^\omega$ for all $q\in (0,\infty]$.
    \item If $\mathcal A_q^\omega\approx \mathcal G_q^\omega$ for some $q\in (0,\infty]$, then $\mathcal B$ is greedy. 
\end{enumerate}
\end{thm}

Our final result is related to partially greedy bases, first introduced by Dilworth et al. \cite{DKKT}

\begin{thm}\label{main2}
	Let $\omega$ be a doubling weight with $i_\omega>0$. Let $\mathcal B$ be a quasi-greedy M-basis in a quasi-Banach space $\mathbb{X}$.
	\begin{enumerate}
    \item If $\mathcal B$ is partially greedy, then $\mathcal{PG}_q^\omega\hookrightarrow \mathcal{G}_q^\omega$ for all $q\in (0,\infty]$.
    \item If $\mathcal B$ is Schauder with Property (I) and Property (W$^*$) and $\mathcal{PG}_q^\omega\hookrightarrow \mathcal{G}_q^\omega$ for some $q\in (0,\infty]$, then $\mathcal B$ is partially greedy. 
\end{enumerate}
\end{thm}

The above-mentioned properties, the dilation index $i_\omega$, and other terminologies  will be defined later. 

\begin{rek}\rm Note that for no basis we have $\mathcal G_q^{\omega}\hookrightarrow \mathcal{PG}_q^{\omega}$. Indeed, for $0<q<\infty$, given $m\in \mathbb N$, we have
	\begin{align*}
		\|e_{m+1}\|_{\mathcal G_q^{\omega}}=\|e_{m+1}\|, 
	\end{align*}
	but 
	\begin{align*}
		\|e_{m+1}\|_{\mathcal{PG}^{\omega}}=\|e_{m+1}\|+\left(\sum_{n=1}^{m}\left(\omega(n)\|e_{m+1}\|\frac{1}{n}\right)^{q}\right)^{\frac{1}{q}}\ge \|e_{m+1}\|\omega(1)\left(\sum_{n=1}^{m}\frac{1}{n}\right)^{\frac{1}{q}}. 
	\end{align*}
	Hence, 
	\begin{align*}
		\frac{\|e_{m+1}\|_{\mathcal G_q^{\omega}}}{\|e_{m+1}\|_{\mathcal{PG}_q^{\omega}}}\le\frac{1}{\omega(1)\left(\sum_{n=1}^{m}\frac{1}{n}\right)^{\frac{1}{q}}} \xrightarrow[m\to \infty]{}0.
	\end{align*}
	A similar computation gives the result for $q=\infty$. 
\end{rek}
\section{Preliminary results}

\subsection{Convexity} 
One of the arguably most important properties of Banach spaces is convexity. In the case of $p$-Banach spaces, we will use the following result that is an extension of classical results of convexity. Given $0<p\le 1$, we put
\begin{equation}\label{eq:fieldconstant1}
\mathbf{A_p}\ =\ \frac{1}{(2^p-1)^{1/p}}.
\end{equation}

\begin{prop}[{\cite[Corollary 2.3]{AABW}}]\label{cor:convexity} Let $\mathbb{X}$ be a $p$-Banach space for some $0<p\le 1$. Let $(x_j)_{j\in J}\subset \mathbb{X}$ with $J$ finite, and $g\in\mathbb{X}$. Then
\begin{enumerate}
	\item For any scalars $(a_{j})_{j\in J}$ with $0\le a_j \le 1$, we have
		\[
		\left\| g+ \sum_{j\in J} a_j x_j \right\| \ \le\  \mathbf{A_p} \sup \left\{ \left\| g+ \sum_{j\in A} x_j\right\| \colon A\subset J\right\}.
		\]
		
	\item For any scalars $(a_{j})_{j\in J}$ with
		$|a_j|\le 1$, we have
		\[
		\left\| g+ \sum_{j\in J} a_j x_j \right\| \ \le\  \mathbf{A_p} \sup\left\{ \left\| g+ \sum_{j\in J} \varepsilon_j x_j\right\| \colon 
		|\varepsilon_j|=1\right\}.
		\] 
	\end{enumerate}
\end{prop}

\subsection{Greedy-type bases}
In \cite{KT}, Konyagin and Temlyakov introduced the TGA $(G^\pi_m)_{m=1}^\infty$ as we described in the previous section and defined greedy bases in Banach spaces.
\begin{defi}\normalfont
An M-basis $\mathcal B$ in a quasi-Banach space is greedy if
$\gamma_m(x) \asymp \sigma_m(x)$.
\end{defi}
Moreover, they characterized greedy bases in terms of unconditionality and democracy in the context of Banach spaces. In \cite{AABW}, the authors proved the same characterization of greediness for quasi-Banach spaces. Recall that an M-basis is $K$-unconditional if
$$K:=\sup_{A\subset \mathbb{N}, | A|<\infty}\| P_A\|\ <\ \infty,$$
where $P_A(x)=\sum_{n\in A}e_n^*(x)e_n$ is the projection operator. To discuss democracy, we need the indicator sum
$$1_{\varepsilon A}\ =\ 1_{\varepsilon A}[\mathcal B, \mathbb{X}]\ :=\ \sum_{n\in A}\varepsilon_n e_n,$$
where $\varepsilon=(\varepsilon_n)_{n\in A}$ with $|\varepsilon_n|=1$ for all $n\in A$. We use the notation $| \varepsilon|=1$.
\begin{defi}\label{defdemo}\normalfont
	We say that $\mathcal B$ is a super-democratic basis in a quasi-Banach space if there is a positive constant $C$ such that
	\begin{equation}\label{demo}
	\|1_{\varepsilon A}\|\ \le\  C\| 1_{\delta B}\|,
	\end{equation}
	for any $| A|\le| B|<\infty$ and $|\varepsilon| = |\delta| = 1$. 
	The smallest constant verifying \eqref{demo} is denoted by $C_{sd}$ and we say that $\mathcal B$ is $C_{sd}$-super-democratic. If $\varepsilon \equiv \delta\equiv 1$, we say that $\mathcal{B}$ is $C_d$-democratic. 
\end{defi}
Equivalently, to define super-democracy, we can use the democracy functions: for each $m=1,2,\ldots$,
$$h_r(m)\ :=\ \sup_{| A|\le m, |\varepsilon|=1}\|1_{\varepsilon A}\| \mbox{ and } h_l(m)\ :=\ \inf_{| A|\ge m, |\varepsilon|=1}\|1_{\varepsilon A}\|.$$
Then $\mathcal B$ is super-democratic if and only if
$$\sup_{m\ge 1}\frac{h_r(m)}{h_l(m)}\ <\ \infty.$$

\begin{rek}\normalfont
For each $m\in\N$ in a $p$-Banach space,
$$2^{1-1/p}h_r(m)\ \le\ \mathbf h_r(m)\ :=\ \sup_{| A|= m, |\varepsilon|=1}\|1_{\varepsilon A}\| \ \le\ h_r(m).$$
Indeed, if $|A|\le N$, take any $B\subset \mathbb{N}$ such that $A\subset B$ and $|B| = N$. Indeed, we have
\begin{align*}\|1_{\varepsilon A}\|^p&\ =\ \left\|\frac{1}{2}(1_{\varepsilon A}+1_{B\backslash A})+\frac{1}{2}(1_{\varepsilon A}-1_{B\backslash A})\right\|^p\\
&\ \le\ \frac{1}{2^p}\|1_{\varepsilon A}+1_{B\backslash A}\|^p + \frac{1}{2^p}\|(1_{\varepsilon A}-1_{B\backslash A})\|^p\ \le\ 2^{1-p}(\mathbf h_r(N))^p.\end{align*}
If $\mathcal B$ is Schauder with basis constant $K_b$,
$$h_l(m)\ \le\ \mathbf{h}_l(m)\ :=\ \inf_{| A|= m, |\varepsilon|=1}\|1_{\varepsilon A}\|\ \le\ K_b h_l(m).$$
\end{rek}

Moreover, Dilworth, Kalton, and Kutzarova \cite{DKK} introduced the concept of semi-greedy bases.

\begin{defi}\normalfont
	We say that $\mathcal B$ is a semi-greedy basis in a quasi-Banach space $\mathbb{X}$ if there is a positive constant $C$ such that
	\begin{equation}\label{semi}
	\vartheta_m(x)\ \le\ C\sigma_m(x), \forall x\in\mathbb{X}, \forall m\in\N.
	\end{equation}
	The smallest constant verifying \eqref{semi} is denoted by $C_{sg}$, and we say that $\mathcal B$ is $C_{sg}$-semi-greedy.
\end{defi}

Semi-greedy Schauder bases were first characterized in Banach spaces with finite cotype in terms of quasi-greediness and democracy \cite{DKK}, and later on, the first author of this paper \cite{B} (see also \cite{B1}) removed the condition of finite cotype. The notion of quasi-greediness was introduced in \cite{KT}: 
 \begin{defi}\normalfont
We say that $\mathcal B$ is a quasi-greedy basis in a quasi-Banach space $\mathbb{X}$ if there is a positive constant $C$ such that
\begin{equation}\label{q-greedy}
\gamma_m(x)\ \le\  C\| x\|, \forall x\in\mathbb{X},\forall m\in\mathbb{N}.
\end{equation}
The smallest constant verifying \eqref{q-greedy} is denoted by $C_{q}$, and we say that $\mathcal B$ is $C_{q}$-quasi-greedy.
\end{defi}

The characterization of semi-greediness in the context of quasi-Banach spaces is unknown and we shall prove that the same characterization  of semi-greediness in Banach spaces holds for quasi-Banach spaces (see Section \ref{semi-g}.)

In the setting of Banach spaces, Dilworth, Kalton, Kutzarova, and Temlyakov \cite{DKKT} introduced partially greedy bases and characterized them as being quasi-greedy and conservative. Bern\'{a} \cite{B0} showed that the same result holds for quasi-Banach spaces (under a stronger notion of partially greediness.)

\begin{defi}\normalfont
We say that a basis $\mathcal B$ in a quasi-Banach space is partially greedy if
$\gamma_m(x) \lesssim \beta_m(x)$.
\end{defi}

\begin{defi}\label{defcons}\normalfont 
We say that $\mathcal B$ is a super-conservative basis in a quasi-Banach space if there is a positive constant $C$ such that
	\begin{equation}\label{cons}
	\|1_{\varepsilon A}\|\ \le\  C\| 1_{\delta B}\|,
	\end{equation}
	for any $| A|\le| B|<\infty$, $\max A < \min B$, and $|\varepsilon| = 1, |\delta| = 1$. 
	The smallest constant verifying \eqref{cons} is denoted by $C_{sc}$ and we say that $\mathcal B$ is $C_{sc}$-super-conservative. If $\varepsilon \equiv \delta\equiv 1$, we say that $\mathcal{B}$ is $C_c$-conservative. 
\end{defi}

In both Definitions \ref{defdemo} and \ref{defcons}, we require $|A|\le |B|$. By $p$-convexity, this requirement can be replaced by $|A| = |B|$ (but the super-democratic (conservative) constants may differ.)

\begin{thm}[Bern\'{a} \cite{B0}]\label{BPG}
A Schauder basis $\mathcal{B}$ of a quasi-Banach space is partially greedy if and only if it is quasi-greedy and is super-conservative. 
\end{thm}

\begin{rek}\normalfont
In \cite{B0}, it was actually proved that an $M$-basis of a quasi-Banach space is strongly partially greedy if and only if it is quasi-greedy and is super-conservative.  In the context of Schauder bases, being strongly partially greedy is equivalent to being partially greedy (see \cite[Remark 3.5]{B0}.)
\end{rek}

\subsection{Weight classes}\label{weights} A weight is any sequence $\omega=(\omega(n))_{n=1}^\infty$ of nonnegative numbers with $\omega(1)>0$. We use the following notation:
	\begin{enumerate}
		\item $\mathbb W$ for the set of nondecreasing and positive weights: $0<\omega(1)\le\omega(2)\le \cdots$, and $\lim_{n\rightarrow\infty}\omega(n)=\infty$.
		\item $\mathbb{W}_d$ is the set of doubling weights, i.e., $\omega\in\mathbb W$ and there is a positive constant $\theta\ge 1$ such that $\omega(2n)\le \theta\cdot\omega(n)$ for all $n\in\mathbb N$. We say that $\omega$ is $\theta$-doubling.
	\end{enumerate}

Associated with a weight $\omega$, we consider the summing weight (see \cite{BBGHO}): for each $m=1,2,\ldots$,
$$\widetilde{\omega}(m)\ :=\ \sum_{n=1}^m\frac{\omega(n)}{n}.$$

\begin{rek}\label{doub}\normalfont
Thanks to \cite[Proposition 2.4]{BBGHO}, we know that if $\omega\in \mathbb{W}_d$, then $\widetilde{\omega}\in\mathbb W_d$. Specifically, if $\omega$ is $\theta$-doubling, then $\widetilde{\omega}$ is $\frac{3\theta}{2}$-doubling.
\end{rek}

A desired relation between a weight $\omega$ and its summing weight $\widetilde{\omega}$ is that $\omega(n)\asymp \widetilde{\omega}(n)$. To obtain the sufficient condition for this relation to hold, we need to define the so-called dilation indices.

\begin{defi} \label{dilation}\normalfont
 The lower and upper dilation sequences associated with a positive sequence $\omega=(\omega(n))_{n=1}^\infty$ are given by
\begin{equation}\varphi_\omega(M) \ :=\ \inf_{k\ge 1} \frac{\omega(Mk)}{\omega(k)} \mbox{ and } \Phi_\omega(M)\ :=\ \sup_{k\ge 1} \frac{\omega(Mk)}{\omega(k)},  M=1,2,3, \ldots.
		\label{phiM2}\end{equation}
\end{defi}
Observe that $\varphi_\omega(M)\le \frac{\omega(M)}{\omega(1)}\le\Phi_\omega(M)$ and
\begin{equation}
\varphi_\omega(M_1)\varphi_\omega(M_2)\ \le\ \varphi_\omega(M_1M_2)\ \le\ \Phi_\omega(M_1M_2)\ \le\ \Phi_\omega(M_1)\Phi_\omega(M_2),
\label{submult}\end{equation}
that is, $\varphi_\omega$ is super-multiplicative and $\Phi_\omega$ is sub-multiplicative.
If $\omega\in\mathbb{W}$, then $\varphi_\omega$ and $\Phi_\omega$ are nondecreasing and 
$\varphi_{\omega}(M)\ge 1$ for all $M\in\mathbb{N}$.

\begin{prop} \label{doub1} Let $\omega\in \mathbb{W}$. The following statements are equivalent:
\begin{enumerate}
	\item  $\omega\in \mathbb{W}_d$.
	\item  $\Phi_\omega (M) < \infty$ for all $M\in \mathbb{N}$.
	\item $\Phi_\omega (M_0) < \infty$ for some $M_0\in \mathbb N_{\ge 2}$.
\end{enumerate}
\end{prop}

\begin{proof} The implications (1) $\Rightarrow$ (2)  $\Rightarrow$ (3) are obvious.
	To prove (3) $\Rightarrow$ (2), fix $M\in \mathbb{N}$. Choose $k\in \mathbb{N}$ such that $M \le M_0^k$. 
	Since $\Phi_\omega$ is nondecreasing and sub-multiplicative, we have
	$$\Phi_\omega(M)\ \le\ \Phi_\omega(M_0^k)\ \le\  \Phi_\omega(M_0)^k \ <\ \infty.$$
	Now, we show that (2) implies (1). Taking into account that
	$$\dfrac{\omega(2k)}{\omega(k)}\leq \Phi_\omega (2)<\infty,$$
	we have the inequality $\omega(2k)\leq \Phi_\omega(2)\omega(k),$ so $\omega\in\mathbb W_d$ and the proof is done.
\end{proof}

\begin{defi} \label{dilation2}\normalfont
	If  $\omega\in\mathbb W$, we define
	the associated lower and upper dilation indices, respectively, by
	\begin{equation} \label{iIeta}
	i_\omega \ =\ \sup_{M>1} \frac{\ln (\varphi_\omega(M))}{\ln M} \qquad \mbox{and}
	\qquad I_\omega \ =\ \inf_{M>1} \frac{\ln (\Phi_\omega(M))}{\ln M}.\end{equation}
\end{defi}

Observe from \eqref{submult} that 
$$(\varphi_\omega(M))^n\ \le\ \varphi_\omega(M^n)\ \le\ \Phi_\omega(M^n)\ \le\ (\Phi_\omega(M))^n,$$
and therefore,
\begin{equation}
\frac{\ln (\varphi_\omega(M))}{\ln M}\ \le\ \frac{\ln (\varphi_\omega(M^n))}{\ln M^n}\ \le\ \frac{\ln (\Phi_\omega(M^n))}{\ln M^n}
\ \le\ \frac{\ln (\Phi_\omega(M))}{\ln M}.
\label{lnMj}\end{equation}
Hence, one can replace ``sup'' and ``inf'' in \eqref{iIeta} by ``$\limsup$'' and ``$\liminf$.''

\begin{prop}\label{doubling}
	If $\omega\in\mathbb{W}$, then $0\le i_\omega\le I_\omega\le \infty$, and 
	\begin{equation}
	i_\omega \ =\ \lim_{M\rightarrow \infty} \frac{\ln (\varphi_\omega(M))}{\ln M} \mbox{ and }
	I_\omega = \lim_{M\rightarrow \infty} \frac{\ln (\Phi_\omega(M))}{\ln M}.\label{iIlim}\end{equation}
	Moreover,  $\omega\in\mathbb{W}_d$ if and only if $I_\omega<\infty$.
\end{prop}
\begin{proof}
	Note that $\omega\in\mathbb{W}$ implies that $1\le \varphi_\omega\le \Phi_\omega$.
	Then \eqref{iIlim} follows easily from \eqref{lnMj}. Now, we prove the last assertion. Assuming that $I_\omega<\infty$, we have $$\ln(\Phi_\omega(M))\leq I_\omega,\; M=1,2,...$$
	Thus, for every $M$, $\Phi_\omega(M)<\infty$ and invoking Proposition \ref{doub1}, $\omega\in\mathbb W_d$. For the converse, if $\omega\in\mathbb W_d$, by Proposition \ref{doub1}, $ \Phi_\omega(M)<\infty$, hence $\frac{\ln(\Phi_\omega(M))}{\ln(M)}<\infty$ for every natural $M>1$. Hence, $I_\omega<\infty$.º
	
	The last  assertion is a direct consequence of Proposition
	\ref{doub1}: 
\end{proof}

\begin{lem}\label{power}
If $\omega\in\mathbb{W}_d$ with doubling constant $\theta$ and $i_\omega>0$, then for any $q\in (0,\infty)$, $\omega^q\in \mathbb{W}_d$ with doubling constant $\theta^q$ and $i_{\omega^q}>0$.
\end{lem}

\begin{proof}
Trivially, $i_{\omega^q}>0$. Also, since $\omega(2n)\le \theta \omega(n)$, it follows that $\omega^q(2n)\le \theta^q \omega^q (n)$ for every $n=1,2,\ldots$. 
\end{proof}

\begin{prop}{\cite[Proposition 2.4]{BBGHO}}\label{one}
Let $\omega\in\mathbb{W}_d$ with constant $\theta$. Then $$\omega(N)\ \le\ \frac{\theta}{\ln(2)}\widetilde{\omega}(N), N=1,2,\ldots.$$
\end{prop}

\begin{prop}{\cite[Proposition 2.5]{BBGHO}}\label{two}
Let $\omega\in\mathbb W$. Then, $\sup_{N\ge 1}\frac{\widetilde{\omega}(N)}{\omega(N)}<\infty$ if and only if $i_\omega>0$.
\end{prop}

\begin{prop}\label{equiv}
Let $\omega\in\mathbb{W}_d$ with $i_\omega>0$. Then $\omega (N)\asymp \widetilde{\omega}(N)$.
\end{prop}
\begin{proof}
	This follows immediately from Propositions \ref{one} and \ref{two}.
\end{proof}

Next, we prove an auxiliary result that shall be used in due course.
\begin{lem}\label{auxi}
Let $\omega\in\mathbb{W}_d$ and $\alpha>I_\omega$. There exists $C_\alpha>0$ such that $$\omega(Mk)\ \le\ C_\alpha M^\alpha \omega(k), \forall M, k\ge 1.$$
\end{lem}
\begin{proof}
Let  $\omega\in\mathbb{W}_d$ and $\alpha>I_\omega$. By \eqref{iIlim}, there exists $M_\alpha\ge 2$ such that
$$\frac{\ln(\Phi_\omega(M))}{\ln(M)}\ \le\ \alpha, \forall M\ge M_\alpha.$$
Therefore, $\Phi_\omega(M)\le M^\alpha$ for all $M\ge M_\alpha$. Now, by definition of $\Phi_\omega(M)$, one has 
$$\omega(Mk)\ \le\ M^\alpha \omega(k), \forall M\ge M_\alpha, k\ge 1.$$
For $M<M_\alpha$ and $k\ge 1$,
$$\omega(Mk)\ \le\ \omega(M_\alpha k)\ \le\ M_\alpha^\alpha\omega(k)\ \le\ M_\alpha^\alpha M^\alpha \omega(k).$$
\end{proof}

\begin{rek}\normalfont
In Section \ref{LRP-URP}, we will give a relation between the dilation indices and the so-called Upper and Lower Regularity Properties. The relation is useful in proving our main theorems.
\end{rek}

\subsection{The truncation operator}
For each $x\in \mathbb{X}$ and each finite  set $A\subset\mathbb{N}$, we define the restricted truncation operator and the truncation operator as follows:
\begin{align*}
\mathcal{U}(x,A) &\ :=\ \min_{n\in A} |e_n^*(x)| \sum_{n\in A} \sgn (e_n^*(x))  e_n,\\
\mathcal{T}(x,A) &\ :=\ \mathcal{U}(x,A)+P_{A^c}(x).
\end{align*}
If $A$ is empty or is infinite, then we use the convention that $\mathcal{U}(x, A) = 0$.

Similar operators were introduced in \cite{DKK}.
Write the quantity:
$$\Gamma\ =\ \sup\{ \| \mathcal{U}(x,A)\| : A \mbox{ is a greedy set of } x, \| x\|\le 1\},$$
$$\Upsilon\ =\ \sup\{ \| \mathcal{T}(x,A)\| : A \mbox{ is a greedy set of } x, \| x\|\le 1\}.$$

For quasi-Banach spaces, the authors in \cite{AABW} proved the following result regarding the boundedness of these operators.

\begin{thm}{\cite[Theorem 4.12 and Theorem 4.13]{AABW}}\label{trop} Let $\mathcal{B}$ be a $C_q$-quasi-greedy basis of a quasi-Banach space $\mathbb{X}$. Then the restricted truncation operator is uniformly bounded, i.e., $\Gamma <\infty$. Also, if $\mathbb{X}$ is $p$-Banach, $$\Gamma\ \le\  C_{q}^2  \eta_p(C_{q}),\mbox{ and } \Upsilon\ \le\ (C_{q}^p+\Gamma^p)^{1/p},$$
	where, if $u>0$,
	$$\eta_p(u)\ :=\ \min_{0<t<1} (1-t^p)^{-1/p} (1-(1+ \mathbf{A_p}^{-1}u^{-1} t)^{-p})^{-1/p}.$$
\end{thm}


\section{Semi-greedy Schauder bases in quasi-Banach spaces}\label{semi-g}

The first author \cite{B} showed that a Schauder basis in a Banach space is semi-greedy if and only if it is quasi-greedy and super-democratic. We generalized the result to the setting of quasi-Banach spaces.

\begin{thm}\label{ch}
	Assume that $\mathcal{B}$ is a Schauder basis in a quasi-Banach space. Then $\mathcal B$ is quasi-greedy and super-democratic if and only if $\mathcal B$ is semi-greedy. Moreover, if $\mathbb{X}$ is a $p$-Banach space,
	$$C_{sd}\ \le\  K_b(1+K_b)C_{sg}^2,$$
	$$C_q\ \le\  K_bC_{sg}(1+(1+K_b)^pC_{sg}^p)^{1/p},$$
$$C_{sg}\ \le\ (2(C_{q}^2\eta_p(C_q))^p+(2\mathbf{A_p}C_{sd}C_{q}^2\eta(C_{q}))^p)^{1/p}.$$
\end{thm}

\begin{proof}
	Assume that $\mathcal{B}$ is $C_{q}$-quasi-greedy and $C_{sd}$-super-democratic. To show the semi-greediness, we proceed as in \cite{DKK}. Take $x\in\mathbb{X}$ and $m\in\mathbb N$. Choose $z=\sum_{n\in B}b_n e_n$ with $| B|=m$, $A$ a greedy set of $x$ with cardinality $m$ and $\varepsilon\equiv \sgn(e_n^*(x-z))$. Set $\alpha:=\max_{n\not\in A}|e_n^*(x)|$ and $\Delta_\alpha:= \{ n : |e_n^*(x-z)|>\alpha\}$. Thus, the set $\Delta_\alpha$ is a greedy set for $x-z$ and $\Delta_\alpha\subset A\cup B$. Define
	$$h\ :=\ P_{A}(x)-P_{A}(\mathcal{T}(x-z,\Delta_\alpha)).$$
	It is easy to verify that
	$$x-h\ = \ \mathcal{T}(x-z,\Delta_\alpha) + P_{B\setminus A}(x-\mathcal{T}(x-z,\Delta_\alpha)).$$
	Let $\Lambda$ be a greedy set of $x-z$ with $|\Lambda| = |B\backslash A| = |A\backslash B|$ and $\min_{n\in\Lambda}|e_n^*(x-z)|\ge \min_{n\in A\backslash B}|e_n^*(x-z)| \ge \alpha$. On the one hand, Theorem \ref{trop} gives
	\begin{equation}\label{p1}
	\| \mathcal{T}(x-z, \Delta_\alpha) \| \ \le\  2^{1/p}C_q^2  \eta_p(C_{q})\| x-z\|.
	\end{equation}

	\noindent On the other hand, by Proposition \ref{cor:convexity}, super-democracy, and Theorem \ref{trop},
	\begin{align}\label{p1.1}
	\nonumber\| P_{B\setminus A}(x-\mathcal{T}(x-z,\Delta_\alpha))\| &\ \le\  2\mathbf{A_p}\alpha \sup_{|\eta|=1}\| 1_{\eta (B\setminus A)}\|\\
	\nonumber&\ \le\  2\mathbf{A_p}C_{sd}\min_{n\in \Lambda}| e_n^*(x-z)| \| 1_{\varepsilon \Lambda}\|\\
	&\ \le\ 2\mathbf{A_p}C_{sd}C_{q}^2\eta_p(C_{q})\| x-z\|.
	\end{align}
	
	\noindent Hence, 
	$$\| x-h\|^p \ \le\  \left[2(C_{q}^2\eta_p(C_q))^p+(2\mathbf{A_p}C_{sd}C_{q}^2\eta(C_{q}))^p\right]\|x-z\|^p.$$
	Thus, since this works for any finite greedy set $A$ of $x$, the basis is $C_{sg}$-semi-greedy with
	$$C_{sg}\ \le\ (2(C_{q}^2\eta_p(C_q))^p+(2\mathbf{A_p}C_{sd}C_{q}^2\eta(C_{q}))^p)^{1/p}.$$
	
	Assume now that $\mathcal{B}$ is $C_{sg}$-semi-greedy. First, we prove that $\mathcal{B}$ is super-democratic. Take $A, B$ such that $| A| \le | B|$, $|\varepsilon|=|\delta|=1$, $C>(A\cup B)$ with $| C| =| A|$. Define the element $y:=1_{\varepsilon A}+1_C$. Hence, if $m=| C|$, then $C$ is a greedy set of $y$ of order $m$. Applying the Chebyshev Greedy Algorithm for a corresponding greedy ordering $\pi$, 
	$$y-{CG}_m^\pi(y)\ =\ 1_{\varepsilon A}+\sum_{n\in C}c_ne_n,$$ 
	for some scalars $(c_n)\subset \mathbb{F}$. Hence, using the basis constant and semi-greediness, we have
	$$\| 1_{\varepsilon A}\|\ \le \ K_b\| y-{CG}_m^\pi(y)\|\ \le\ K_bC_{sg}\sigma_m(y)\ \le\ K_bC_{sg}\| 1_C\|.$$
	
	To estimate $\|1_{\varepsilon B}\|$, we define the element
	$z:=1_C+1_{\delta B}$. Arguing as before, if $m=| B|$, then $B$ is a greedy set corresponding to a greedy ordering $\pi'$ and 
	$$\| 1_C\|\ \le\ (1+K_b)\| z-{CG}_m^{\pi'}(z)\|\ \le\  (1+K_b)C_{sg}\sigma_m(z)\ \le\ (1+K_b)C_{sg}\| 1_{\delta B}\|.$$
	Thus, $\mathcal{B}$ is $C_{sd}$-super-democratic with
	$$C_{sd}\ \le\ K_b(1+K_b)C_{sg}^2.$$
	
	Next, we prove that $\mathcal{B}$ is quasi-greedy. Let $x\in\mathbb{X}$ with finite support, $A$ be finite greedy set of cardinality $m$ of $x$ and $C>\supp(x)$ such that $| C|=m$. Define the element $z:=(x-P_{A}(x))+\alpha1_C$, with $\alpha=\max_{n\not\in A}| e_n^*(x)|$. Hence, $C$ is a greedy set of $z$, and, there is a greedy ordering $\pi$ such that
	$$z-{CG}_m^\pi(z)\ =\ (x-P_A(x))+\sum_{n\in C} d_n e_n,$$
	for some $(d_n)\subset\mathbb{F}$.
	Thus,
	\begin{align*}
	\| x-P_{A}(x)\|^p&\ \le\ K_b^p\| z-{CG}_m^\pi(z)\|^p\\
	 &\ \le \  K_b^pC_{sg}^p\| x+\alpha1_C\|^p\\
	 &\ \le \ K_b^pC_{sg}^p(\| x\|^p + \| \alpha1_C\|^p).
	\end{align*}
	
	If we consider the element $y:=x+\alpha1_C$, a greedy set of $y$ is $A$. Hence, for some greedy ordering $\pi'$,
	$y-{CG}_m^{\pi'}(y)=(x-P_{A}(x))+\sum_{n\in A}a_ne_n+\alpha1_C$. Then
$$\| \alpha1_C\| \ \le\  (1+K_b)\| y-{CG}_m^{\pi'}(y)\| \ \le\  (1+K_b)C_{sg}\sigma_m(y)\ \le\ (1+K_b)C_{sg}\| x\|.$$
	Therefore, the basis is $C_q$-quasi-greedy with $C_{q}\le K_bC_{sg}(1+(1+K_b)^pC_{sg}^p)^{1/p}.$
\end{proof}

\section{Approximation classes and (semi-) greedy bases}
Before provingTheorem \ref{main1}, we need the following technical propositions that generalize \cite[Proposition 7.1]{GHN} and \cite[Lemma 3.3]{W}.

\begin{prop}\label{imp2}
Let $\omega\in\mathbb W_d$. Let $\mathcal B$ be a basis of a quasi-Banach space and $f, g$ be two nondecreasing functions with $f, g: \N\rightarrow (0,\infty)$, $g$ doubling with constant $d$, and
\begin{equation}\label{m1}
\limsup_{n\rightarrow\infty}\frac{f(n)}{g(n)} \ =\  \infty.
\end{equation}
Then there exist $\eta_j \ge k_j \ge 1$, $j=1,2,\ldots$, such that
$$\lim_{j\rightarrow\infty}\frac{\eta_j}{k_j}=\infty\mbox{ and } \frac{f(k_j)}{g(\eta_j)}\ \ge\  \frac{\omega(\eta_j)}{\omega(k_j)}.$$
\end{prop}

\begin{proof}
	By \eqref{m1}, there exists an increasing sequence $(z_n)_{n=1}^\infty$ with $z_n\rightarrow \infty$ such that
	\begin{equation}\label{m2}
	\lim_{n\rightarrow \infty}\frac{f(z_n)}{g(z_n)} \ =\ \infty.
	\end{equation}
Given $z_n$, define $r:\mathbb{N}\rightarrow\mathbb{N}$ such that $2^{r(n)-1}\le z_n<2^{r(n)}$. Since $g$ is doubling, for any $n, M\in\N$,
\begin{equation}\label{m3}
g(z_n M)\ \le\ g(2^{r(n)}M)\ \le\ d^{r(n)} g(M).
\end{equation}
Since $\omega\in \mathbb{W}_d$, by Proposition \ref{doubling}, $I_\omega<\infty$, and we can fix $\alpha> I_\omega$. By \eqref{m2}, we take an increasing sequence $(k_j)_{j=1}^\infty$ where each $k_j$ is some $z_{n}$ such that
\begin{equation}\label{m4}
\frac{f(k_j)}{g(k_j)}\ \ge\  d^{r(j)}C_\alpha z_j^{\alpha},
\end{equation}
where $C_\alpha$ is as in Lemma \ref{auxi}. Using Lemma \ref{auxi} and \eqref{m4}, we obtain
\begin{equation}\label{m5}
\frac{f(k_j)}{g(k_j)}\ \ge\  d^{r(j)}C_\alpha z_j^\alpha \ \ge\ d^{r(j)}\frac{\omega(z_j k_j)}{\omega(k_j)}.
\end{equation}
Define $\eta_j = z_j k_j$. Since $z_j\rightarrow \infty$,
$$\lim_{j\rightarrow\infty} \frac{\eta_j}{k_j}\ =\ \infty,$$
so the first part of this proposition is proved. For the second,
$$\frac{f(k_j)}{g(\eta_j)}\ =\ \frac{f(k_j)}{g(z_j k_j)}\ \stackrel{\eqref{m3}}{\ge}\ \frac{f(k_j)}{d^{r(j)}g(k_j)}\ \stackrel{\eqref{m5}}{\ge}\ \frac{\omega(z_j k_j)}{\omega(k_j)}\ =\ \frac{\omega(\eta_j)}{\omega(k_j)}.$$
The proof is done.
\end{proof}

\begin{defi}\normalfont
Let $\mathcal{B}$ be a basis in a quasi-Banach space. We say that $\mathcal{B}$ has Property (W) if there exists a positive 
constant $C$ such that for all $n\in \mathbb{N}$ and for all $m\ge n$, there exist $A\subset\mathbb{N}_{> m}$ with $|A| = n$ and $|\varepsilon|=1$ such that 
\begin{equation}\label{w}h_r(n) \ \le\ C\|1_{\varepsilon A}\|.\end{equation}
The smallest constant in \eqref{w} is denoted by $\mathcal{K}$, and we say that $\mathcal{B}$ has the $\mathcal{K}$-Property (W). 
\end{defi}

\begin{prop}\label{superdemoimpliesw}
If a basis $\mathcal{B}$ in a $p$-Banach space is $C_{sc}$-super-conservative, then $\mathcal{B}$ has $2^{1/p}C_{sc}$-Property (W). 
\end{prop} 

\begin{proof}
Fix $n\in \mathbb{N}$ and $m\ge n$.  Choose $B$ and $|\varepsilon|=1$ such that $|B| = n$ and
$h_r(n)\le 2^{1/p}\|1_{\varepsilon B}\|$. Let $A = \{m+\max B + 1,\ldots, m+\max B + n\}$. Since $\mathcal B$ is $C_{sc}$-super-conservative, we have $$h_r(n)\ \le\ 2^{1/p}\|1_{\varepsilon B}\|\ \le\ 2^{1/p}C_{sc}\|1_A\|.$$
Then, $\mathcal B$ has $2^{1/p}C_{sc}$-Property (W).
\end{proof}

\begin{exa}\normalfont
All bases in Section \ref{listexa} have Property (W). Subsection \ref{rex1} gives an example of a non-conservative basis that has Property (W).
\end{exa}

\begin{prop}\label{imp1}
Let $\mathcal B$ be a Schauder basis with $\mathcal K$-Property (W) in a $p$-Banach space $\mathbb X$. Let $\omega\in\mathbb W_d$ with $i_\omega>0$. Assume that there exist two sequences of integers $\eta_j \ge k_j\ge 1$, $j=1,2,\ldots$, such that
\begin{equation}\label{condi}
\lim_{j\rightarrow\infty}\frac{\eta_j}{k_j}\ =\ \infty\mbox{ and } \frac{h_r(k_j)}{h_l(\eta_j)}\ \ge\ \frac{\omega(\eta_j)}{\omega(k_j)}.
\end{equation}
Then $\mathcal A_q^\omega \hookrightarrow \mathcal{CG}_q^\omega$ does not hold for any $q\in (0,\infty]$.
\end{prop}

\begin{proof}
Assume first that $0<q<\infty$. For each $j\in\mathbb N$, choose $\Gamma_{l,j}\subset\mathbb{N}$ with $|\Gamma_{l,j}| = \eta_j$ and $|\varepsilon|=1$ such that
$$\mathbf{h}_l(\eta_j)\ :=\ \inf_{|A|=\eta_j, |\delta|=1}\|1_{\delta A}\|\ \ge\ \frac{\|1_{\varepsilon \Gamma_{l,j}}\|}{2}.$$
Since $h_l(\eta_j)\le \mathbf{h}_l(\eta_j)\le K_bh_l(\eta_{j})$,
\begin{equation}\label{n1}
 h_l(\eta_j)\ \gtrsim \ \|1_{\varepsilon \Gamma_{l,j}}\|.
\end{equation}
By Property (W), there exist $|\delta|=1$ and $\Gamma_{r,j}$ with $|\Gamma_{r,j}|=k_j\le \eta_j$ such that $\Gamma_{r,j}>\Gamma_{l,j}$ and
\begin{equation}\label{n2}
h_r(k_{j})\ \le\ \mathcal K\|1_{\delta\Gamma_{r,j}}\|.
\end{equation}

\noindent Define the element $x_j := 2\cdot 1_{\varepsilon\Gamma_{l,j}} + 1_{\delta \Gamma_{r,j}}$. We have
\begin{equation}\label{n3}
\| x_j\|^p\ \le\ 2^p\| 1_{\varepsilon\Gamma_{l,j}}\|^p + \| 1_{\delta\Gamma_{r,j}}\|^p \ \stackrel{\eqref{n1}}{\le} \ 2^{2p}K_b^p(h_l(\eta_j))^p + (h_r(k_j))^p .
\end{equation}

\noindent Since $\omega$ is nondecreasing and $\eta_j \ge k_j$, $\omega(\eta_j)\ge\omega(k_j)$. Hence, using our hypothesis,
\begin{equation}\label{n4}
h_l(\eta_j)\ \le\  h_r(k_j)\frac{\omega(k_j)}{\omega(\eta_j)}\ \le\ h_r(k_j).
\end{equation}
By \eqref{n3} and \eqref{n4}, we obtain
\begin{equation}\label{n5}
\| x_j\| \ \lesssim\  h_r(k_j).
\end{equation}

\noindent Now, if $n\in\{ 1,\ldots,\eta_j\}$, a greedy set of $x_j$ of order $n$ is a subset of $\Gamma_{l, j}$; hence, 
\begin{equation}\label{n6}
\vartheta_n(x_j)\ \ge\ \frac{\|1_{\delta\Gamma_{r,j}}\|}{1+K_b}\ \ge\ \frac{h_r(k_j)}{\mathcal K(1+K_b)}.
\end{equation}
Hence, 
\begin{align}\label{n7}
\nonumber\| x_j\|_{\mathcal{CG}_q^\omega}&\ \stackrel{\eqref{n6}}{\gtrsim}\  \left(\sum_{n=1}^{\eta_j}\frac{1}{n}\left(\omega(n)h_r(k_j)\right)^q\right)^{1/q}\\
&\ =\ h_r(k_j)(\widetilde{\zeta}(\eta_j))^{1/q},
\end{align}
where $\zeta(j)=(\omega(j))^q$ and $\widetilde{\zeta}$ is the summing weight corresponding to $\zeta$. Applying Lemma \ref{power} and Corollary \ref{equiv}, we know that 

\begin{equation}\label{aux}
\zeta(n)\ \lesssim \ \widetilde{\zeta}(n)\ \lesssim \ \zeta(n).
\end{equation}
Thus,
\begin{equation}\label{n7.1}
\|x_j\|_{\mathcal{CG}_q^\omega}\ \gtrsim \ h_r(k_j)\omega(\eta_j)
\end{equation}

Regarding the error $\sigma_n(x_j)$, for any $1\le n\le\eta_j+k_j$,
\begin{equation}\label{n8}
\sigma_n(x_j)\ \le\  \| x_j\| \ \stackrel{\eqref{n5}}{\lesssim} \ h_r(k_j),
\end{equation}
but if $n > k_j$, we get
\begin{equation}\label{n9}
\sigma_n(x_j)\ \le\  2\|1_{\varepsilon \Gamma_{l,j}}\|\ \stackrel{\eqref{n1}}{\lesssim}\ h_l(\eta_j).
\end{equation}
Due to \eqref{n5}, \eqref{n8}, and \eqref{n9}, we obtain
\begin{eqnarray}\label{n10}
\nonumber\| x_j\|_{\mathcal A_q^\omega}&\lesssim&  h_r(k_j)+\left(\sum_{n=1}^{k_j}\frac{1}{n}(\omega(n)h_r(k_j))^q+\sum_{n=k_j+1}^{k_j+\eta_j}\frac{1}{n}(\omega(n)h_l(\eta_j))^q\right)^{1/q}\\\nonumber
&\le&  h_r(k_j)+\left(\widetilde{\zeta}(k_j)(h_r(k_j))^q+\widetilde{\zeta}(k_j+\eta_j)(h_l(\eta_j))^q\right)^{1/q}\\\nonumber
&\stackrel{\eqref{aux}, \omega\in \mathbb{W}_d}{\lesssim}&  h_r(k_j)+\left((\mathcal \omega(k_j)h_r(k_j))^q+(\omega(\eta_j)h_l(\eta_j))^q\right)^{1/q}\\\nonumber
&\stackrel{\mbox{hypothesis}}{\le}& h_r(k_j)+\left((\omega(k_j)h_r(k_j))^q+(\omega(k_j)h_r(k_j))^q\right)^{1/q}\\
&\lesssim& h_r(k_j)\omega(k_j).
\end{eqnarray}
Writing $\eta_j = s_j k_j$ and  using our hypothesis, we have that $s_j\rightarrow \infty$, and applying Theorem \ref{better} gives
\begin{equation}\label{n13}
\frac{\omega(\eta_j)}{\omega(k_j)}\ =\ \frac{\omega(s_jk_j)}{\omega(k_j)}\ \ge\  C_\alpha(s_j)^\alpha\ \rightarrow\ \infty,
\end{equation}
where $\alpha$ is positive and strictly smaller than $i_\omega$. By \eqref{n7.1}, \eqref{n10}, and \eqref{n13}, 
$$\frac{\| x_j\|_{\mathcal{CG}_q^\omega}}{\| x_j\|_{\mathcal A_q^\omega}}\ \gtrsim\ \frac{h_r(k_j)\omega(\eta_j)}{h_r(k_j)\omega(k_j)}\\
\ = \frac{\omega(\eta_j)}{\omega(k_j)}\ \rightarrow\ \infty\mbox{ as } j\rightarrow\infty.$$

Consider now $q=\infty$. By \eqref{n6},
\begin{equation}\label{n14}
\| x_j\|_{\mathcal{CG}_\infty^\omega} \ \ge \ \sup_{1\le n\le \eta_j}\omega(n)\vartheta_n(x_j)\ \gtrsim \ \omega(\eta_j)h_r(k_j).
\end{equation}
Arguing as in \eqref{n10}, we obtain
\begin{eqnarray}\label{n15}
\| x_j\|_{\mathcal A_\infty^\omega}&\stackrel{\eqref{n5}}{\lesssim}&  h_r(k_j)+\sup_{1\le n\le k_j}\omega(n)\sigma_n(x_j)+\sup_{k_j < n\le k_j+\eta_j}\omega(n)\sigma_n(x_j)\nonumber \\
&\stackrel{\eqref{n8},\eqref{n9}}{\lesssim}&  h_r(k_j) + \omega(k_j)h_r(k_j)+ \omega(\eta_j)h_l(\eta_j)\nonumber\\
&\stackrel{\mbox{hypothesis}}{\lesssim}&  h_r(k_j)\omega(k_j).
\end{eqnarray}
We conclude that
\begin{equation*}
\frac{\| x_j\|_{\mathcal{CG}_\infty^\omega}}{\| x_j\|_{\mathcal A_\infty^\omega}}\ \gtrsim\   \frac{h_r(k_j)\omega(\eta_j)}{h_r(k_j)\omega(k_j)}\ =\ \frac{\omega(\eta_j)}{\omega(k_j)}\ \rightarrow\ \infty\mbox{ as }j\rightarrow\infty.
\end{equation*}
This completes our proof.
\end{proof}

\begin{prop}\label{casec}
Let $\mathcal B$ be a Schauder basis in a $p$-Banach space and $\omega\in\mathbb{W}_d$ with $i_\omega>0$. If $\mathcal A_q^\omega \hookrightarrow \mathcal{CG}_q^\omega$ for some $q\in (0,\infty]$, then $h_l$ is a doubling function.
\end{prop}

\begin{proof}
Assume that $q\in (0,\infty)$ and that $h_l$ is not doubling. 

Step 1: set up. Take sufficiently large $s>1$. Then there exists $n_s\in\mathbb{N}$ such that
	$$sh_l(n_s)\ \le\  h_l(2n_s).$$
	Choose a set $M_s$ with $|M_s| = n_s$ and $|\varepsilon|=1$ such that
	$$\| 1_{\varepsilon M_s}\|^p-\frac{1}{s^p}\ \le\  \mathbf{h}_l^p(n_s)\ \le\  K_b^p h^p_l(n_s) \ \le\  K_b^p\| 1_{\varepsilon M_s}\|^p.$$
    Let  $D\subset\mathbb{N}$ such that $M_s < D$ and $|D| = n_s$. Then
	\begin{equation*}
		s^p\left(\| 1_{\varepsilon M_s}\|^p-\frac{1}{s^p}\right)\ \le\  (K_b s)^ph_l^p(n_s)\ \le\  K_b^p h^p_l(2n_s)\ \le\  K_b^p(\| 1_{\varepsilon M_s}\|^p+\| 1_{D}\|^p).
	\end{equation*}
	Hence,
	\begin{equation}\label{onee}
	\left(s^p-K_b^p\right)\| 1_{\varepsilon M_s}\|^p-1\ \le\  K_b^p\| 1_{D}\|^p.
	\end{equation}
	Partition  $D=\cup_{i=1}^r V_i$ with $r=\lfloor s^{p/2}\rfloor$, and each set $V_i$ has cardinality $\lceil n_s/r\rceil$ or $\lfloor n_s/r\rfloor$. Let $V_s$ be a set in the partition such that 
	$\|1_{V_s}\|\ =\ \max_{1\le i\le r}\|1_{V_i}\|$. Dividing each term of \eqref{onee} by $r$, we obtain
	\begin{equation*}\label{on}
		\frac{\left(s^p-K_b^p\right)}{r}\| 1_{\varepsilon M_s}\|^p\ \le\ K_b^p\| 1_{V_s}\|^p + \frac{1}{r}.
	\end{equation*}
	Thus,
	\begin{equation}\label{twwo}
	\| 1_{\varepsilon M_s}\|^p \ \le\ \frac{1+rK_b^p}{s^p - K_b^p}\| 1_{V_s}\|^p.
	\end{equation}
	For sufficiently large $s$, \eqref{twwo} implies that $\| 1_{\varepsilon M_s}\| \le \| 1_{V_s}\|$. Define $x_s:= 2\cdot 1_{\varepsilon M_s}+1_{V_s}$. Then 
	\begin{equation}\label{three}
	\| x_s\|^p\ \le\ \|1_{V_s}\|^p + 2^p\| 1_{\varepsilon M_s}\|^p\ \le\ \left(1+2^p\right)\| 1_{V_s}\|^p.
	\end{equation}

	Step 2: bound $\| x_s\|_{\mathcal{CG}_q^\omega}$. If $k\le | M_s| = n_s$, a greedy set of $x_s$ is a subset $M_{1,s}$ of $M_s$ and so,
	\begin{equation}\label{four}
	\|1_{V_s}\|\ \le\ (1+K_b)\left\| 1_{V_s}+2\cdot 1_{\varepsilon(M_s\backslash M_{1,s})}+\sum_{n\in M_{1,s}}a_n e_n\right\| \ \le\ (1+K_b)\vartheta_k(x_s),
	\end{equation}
	for some $(a_n)\subset\mathbb F$. Let $\zeta(j)=(\omega(j))^q$ and $\widetilde{\zeta}$ is the summing weight corresponding to $\zeta$. Applying lemma \ref{power} and Proposition \ref{equiv}, we know that 
\begin{equation}\label{zeta}
\zeta(n)\ \lesssim \ \widetilde{\zeta}(n)\ \lesssim \ \zeta(n).
\end{equation}
	Hence, for $0<q<\infty$, 
	\begin{equation}\label{five}
    \| x\|_{\mathcal{CG}_q^\omega}\ \ge\  \left(\sum_{k=1}^{n_s} ( \omega(k)\vartheta_k(x_s))^q\frac{1}{k}\right)^{1/q}\ \stackrel{\eqref{four}}{\gtrsim} \  (\widetilde{\zeta}(n_s))^{1/q}\|1_{V_s}\|\ \stackrel{\eqref{zeta}}{\gtrsim}\  \omega(n_s)\|1_{V_s}\|.
	\end{equation}
	
	Step 3: bound $\| x_s\|_{\mathcal{A}_q^\omega}$. If $k\le | V_s|$,
	\begin{equation}\label{t11}
	\sigma_k(x_s)\ \le\ \| x_s\| \ \stackrel{\eqref{three}}{\lesssim}\  \| 1_{V_s}\|.
	\end{equation}
	
	\noindent If $k > | V_s|$,
	\begin{equation}\label{t1}
	\sigma_k(x_s)\ \le\ 2\|1_{\varepsilon M_s}\|\ \stackrel{\eqref{twwo}}{\lesssim}\ \left(\frac{1+rK_b^p}{s^p - K_b^p}\right)^{1/p}\| 1_{V_s}\|.
	\end{equation}
	Define $C(s,p):= \left(\frac{1+rK_b^p}{s^p - K_b^p}\right)^{1/p}$. Since $i_\zeta > 0$, letting $\alpha=i_\zeta/2$ and using Theorem \ref{better} give
	\begin{equation}\label{zeta2}
	\frac{\zeta(n)}{\zeta(m)}\ \ge\ C_{\alpha}\left(\frac{n}{m}\right)^{\alpha}, \forall m\le n.
	\end{equation}We have
	\begin{eqnarray}\label{p2}
	\nonumber\| x_s\|_{\mathcal A_q^\omega} & =& \| x_s\| + \left(\sum_{k=1}^{| V_s|}( \omega(k)\sigma_k(x_s))^q\frac{1}{k}+\sum_{k=| V_s|+1}^{2n_s}( \omega(k)\sigma_k(x_s))^q\frac{1}{k}\right)^{1/q}\\
	\nonumber & \stackrel{\eqref{three},\eqref{t11},\eqref{t1}}{\lesssim}&  
	\|1_{V_s}\|\left(1+\left(\widetilde{\zeta}(| V_s|)+C(s,p)\widetilde{\zeta}(2n_s)\right)^{1/q}\right)\\
	\nonumber& \stackrel{\eqref{zeta}}{\lesssim}&  \|1_{V_s}\|\left(1+\left(\omega^q(| V_s|)+C(s,p)\omega^q(2n_s)\right)^{1/q}\right)\\
	\nonumber & \stackrel{\omega\in\mathbb{W}_d}{\lesssim}& \|1_{V_s}\|\left(1+\left(\omega^q(| V_s|)+C(s,p) \omega^q(n_s)\right)^{1/q}\right)\\
	\nonumber& =& \|1_{V_s}\|\omega(n_s)\left(\frac{1}{\omega(n_s)}+\left(\frac{\omega^q(| V_s|)}{\omega^q(n_s)}+C(s,p)\right)^{1/q}\right)\\
	\nonumber& \stackrel{\eqref{zeta2}}{\le}& \|1_{V_s}\|\omega(n_s)\left(\frac{1}{\omega(n_s)}+\left(\frac{1}{C_\alpha}\left(\frac{|V_s|}{n_s}\right)^\alpha+C(s,p)\right)^{1/q}\right)\\
	& \stackrel{\eqref{five}}{\lesssim}& \| x_s\|_{\mathcal{CG}_q^\omega}\left(\frac{1}{\omega(n_s)}+\left(\frac{1}{C_\alpha}\left(\frac{|V_s|}{n_s}\right)^\alpha+C(s,p)\right)^{1/q}\right).
	\end{eqnarray}
	Therefore,
	\begin{equation*}
	\frac{\| x_s\|_{\mathcal{CG}_q^\omega}}{\| x_s\|_{\mathcal{A}_q^\omega}}\ \gtrsim\ 
	\left(\frac{1}{\omega(n_s)}+\left(C_\alpha \left(\frac{|V_s|}{n_s}\right)^\alpha+C(s,p)\right)^{1/q}\right)^{-1}\ \rightarrow\ \infty \mbox{ as } s\rightarrow \infty.
	\end{equation*}
	The case when $q = \infty$ is similar. 
	We have that $\mathcal A_q^\omega \hookrightarrow \mathcal{CG}_q^\omega$ does not hold for any $q\in (0,\infty]$.
\end{proof}

\begin{proof}[Proof of Theorem \ref{main1}]
Item (1) follows directly from the definitions of $\mathcal A^{\omega}_q$, $\mathcal{CG}^\omega_q$, and semi-greediness. We prove item (2). By Theorem \ref{ch}, it suffices to prove that $\mathcal{B}$ is super-democratic. Assume otherwise. By Proposition \ref{casec}, $h_l$ is doubling. Apply Proposition \ref{imp2} with $f = h_r$ and $g = h_l$ to obtain sequences $(k_j)$ and $(n_j)$ satisfying \eqref{condi}. Now Proposition \ref{imp1} implies that $\mathcal A_q^\omega \hookrightarrow \mathcal{CG}_q^\omega$ does not hold, which contradicts our hypothesis. 
\end{proof}

\begin{proof}[Proof of Theorem \ref{ghn}] 
Item (1) follows directly from the definitions of $\mathcal A^{\omega}_q$, $\mathcal{G}^\omega_q$, and greediness. We prove item (2) using the exact argument as in the proof of Theorem \ref{main1}. However, unlike Theorem \ref{main1}, we do not need Property (W). In fact, we only have to consider the element $x_j = 1_{\varepsilon\Gamma_{r,j}} + 2\cdot 1_{\varepsilon\Gamma_{l,j}\backslash (\Gamma_{r,j}\cap\Gamma_{l,j})}$ defined in \cite[Proposition 7.1]{GHN} and apply unconditionality instead of the Property (W) in \eqref{n6}.
\end{proof}


\section{Approximation classes and partially greedy bases}

Our goal of this section is to prove Theorem \ref{main2}. Recall that bounding $\sigma_m(x)$ effectively is crucial in the proof of Theorem \ref{main1} for greedy bases. However, for partially greedy bases, establishing an effective bound for $\beta_m(x)$ is considerably more difficult. For example, if $|\supp(x)| = k$, then $\sigma_m(x) = 0$ for all $m > k$, but the same conclusion does not necessarily hold for $\beta_{m}(x)$. Furthermore, we have more freedom in choosing the vector $y$ in \eqref{ee1} to estimate $\sigma_m(x)$, while $S_m(x)$ in the definition of $\beta_m(x)$ is fixed. Hence, to have the equivalences as in Theorem \ref{main2}, we require our bases to satisfy certain properties that allow us to estimate $\beta_m(x)$ more effectively. First, we need some definitions.

Set $\mathbb{D} = \left\{(m,u)\in \mathbb{N}\times \mathbb{N}: m\le u\right\}$.
Define the left and right restricted democracy functions as follows: 

$$h_{R,l}(m,u) \ :=\ \sup_{\substack{|A| = m, \max A \le u\\|\varepsilon| = 1}}\|1_{\varepsilon A}\| \mbox{ and } h_{R,r}(m,u)\ :=\ \inf_{\substack{|A| = m, \min A > u\\ |\varepsilon| = 1}}\|1_{\varepsilon A}\|.$$
where $h_{R,r}(m,u)$ is defined on $\mathbb{N}\times \mathbb{N}$ and $h_{R,l}(m, u)$ is defined on $\mathbb{D}$. 

\begin{prop}\label{rpp1}
\begin{enumerate}
    \item For a Schauder basis $\mathcal{B}$, it holds that \begin{equation}\label{ree60}K_b h_{R,r}(m_1,u)\ \ge\ h_{R,r}(m_2,u),\forall m_1\ge m_2.\end{equation}
    \item For a $C_q$-quasi-greedy basis $\mathcal{B}$, it holds that 
    $$C_q h_{R,l}(m_1, u)\ \ge \ h_{R,l}(m_2, u),\forall u\ge m_1\ge m_2,\forall u\in \mathbb{N}.$$
    \item  There exists a non-quasi-greedy Schauder basis such that 
$$\sup_{m_2\le m_1\le u}\frac{h_{R,l}(m_2, u)}{h_{R,l}(m_1, u)}\ =\ \infty.$$
\end{enumerate}
\end{prop}

\begin{proof}
    (1) Let $A\subset\mathbb{N}$, $|A| = m_1$, $\min A > u$ and $|\varepsilon| = 1$. Choose $B$ to be the set of $m_2$ smallest numbers in $A$. We have $$K_b\|1_{\varepsilon A}\|\ \ge\ \|1_{\varepsilon B}\|\ \ge\ h_{R, r}(m_2, u).$$ Taking the inf over all sets $A$ and $|\varepsilon| = 1$ gives the desired conclusion.
    
    (2) Let $A\subset\mathbb{N}$, $|A| = m_2$, $\max A \le u$ and $|\varepsilon| = 1$. Choose $B$ to be a subset of $\{1, \ldots, u\}$ such that $B\cap A = \emptyset$ and $|B| = m_1-m_2$. Consider $x = 1_{\varepsilon A} + 1_{B}$. Since $B$ is a greedy set of $x$, we obtain
    $$\|1_{\varepsilon A}\|\ =\ \|x-1_{B}\|\ \le\ C_q\|x\|\ \le\ C_qh_{R,l}(m_1, u).$$
    Taking the sup over all sets $A$ and $|\varepsilon| = 1$ finishes the proof. 
    
    (3) See Subsection \ref{diffbasis}.
\end{proof}

\begin{prop}\label{p3}
A basis $\mathcal{B}$ is super-conservative if and only if 
$$\sup_{(m, u)\in\mathbb{D}}\frac{h_{R,l}(m,u)}{h_{R,r}(m,u)} \ < \ \infty.$$
\end{prop}

\begin{proof}
Assume that $\mathcal{B}$ is $C_{sc}$-super-conservative. Fix $(m, u)\in\mathbb{D}$. Choose $A, B\subset\mathbb{N}$ with $|A| = |B| = m$ and $\max A \le u < \min B$. For any $|\varepsilon|=|\delta|=1$, we have
$$\frac{\|1_{\varepsilon A}\|}{\|1_{\delta B}\|}\ \le\ C_{sc},$$
so taking the sup over all $A, |\varepsilon|=1$ and the inf over all $B, |\delta|=1$ gives 
$$\frac{h_{R,l}(m,u)}{h_{R,r}(m,u)}\ \le\  C_{sc}.$$
Taking the sup over $(m, u)\in \mathbb{D}$ completes the proof. 

Now assume that $$\sup_{(m, u)\in\mathbb{D}}\frac{h_{R,l}(m,u)}{h_{R,r}(m,u)} \ < \ C,$$
for some constant $C$.
Choose $|\varepsilon|= |\delta|=1$ and $A, B\subset\mathbb{N}$ with $|A| = |B|, \max A < \min B$. We have
$$\|1_{\varepsilon A}\| \ \le\ h_{R,\ell}(|A|,\max A)\mbox{ and }\|1_{\delta B}\|\ \ge\ h_{R,r}(|B|,\max A).$$
Hence, 
\begin{equation*}
    \frac{\|1_{\varepsilon A}\|}{\|1_{\delta B}\|}\ \le\ \frac{h_{R,l}(|A|, \max A)}{h_{R,r}(|B|, \max A)}\ <\ C.
\end{equation*}
Hence, $\mathcal{B}$ is super-conservative. 
\end{proof}

The next definition expedites our introduction of Property (I).

\begin{defi}\normalfont
A function $\psi:\mathbb{N}\rightarrow\mathbb{N}$ is called a characteristic function of $h_{R,l}$ if $(n, \psi(n))\in \mathbb{D}$ for all $n\in \mathbb{N}$ and 
$$\sup_u h_{R,l}(m, u) \ \lesssim\ h_{R,l}(m, \psi(m)).$$
Clearly, such a function $\psi$ exists but is not unique. 
\end{defi}

\begin{exa}\normalfont
A characteristic function of $h_{R,l}$ is found as follows: let $\psi(m)$ be the smallest integer at least $m$ such that $h_r(m)\le 2h_{R,\ell}(m,\psi(m))$.
\end{exa}

\begin{defi}\normalfont
A basis $\mathcal{B}$ is said to have Property (I) if 
\begin{enumerate}
    \item we have
$$\sup_{\substack{\ell\in \mathbb{N}\cup \{0\}\\ u\in \mathbb{N}}}\frac{h_{R,r}(2^{\ell+1}u, u)}{h_{R, r}(2^{\ell}u, u)} \ <\ \infty,\mbox{ and }$$
\item there is a characteristic function $\psi$ of $h_{R,l}$ such that $h_{R,r}(u,u)\lesssim h_{R, r}(m, u)$ whenever $u\le \psi(m)$.
\end{enumerate}
\end{defi}

\begin{defi}\normalfont
A basis $\mathcal{B}$ is said to have Property (W$^*$) if there is constants $C_1, C_2>0$ such that
for every $m\in \mathbb{N}$, there exist $A\subset\mathbb{N}_{\le C_1m}$ and $|\varepsilon|=1$ such that $|A| = m$ and $\|1_{\varepsilon A}\|\ \le\ C_2h_{l}(m)$. We say $\mathcal{B}$ has $(C_1, C_2)$-Property (W$^*$). 
\end{defi}

\begin{exa}\normalfont
All bases in Section \ref{listexa} have Property (I) and Property (W$^*$). In particular, Subsection \ref{Schreiermodified} gives an unconditional and conservative basis with Property (I) and Property (W$^*$), but the basis is not democratic. Subsection \ref{Schreiermodified} gives an unconditional basis with Property (I) and Property (W$^*$) but is not conservative. 
\end{exa}

\begin{prop}\label{rpp41}
Let $w\in \mathbb{W}_d$. If a Schauder basis $\mathcal{B}$ has property (I) and is not conservative, then there exist $\eta_j\ge u_j\ge k_j\ge 1$, $j=1,2,\ldots$, such that 
$$\lim_{j\rightarrow\infty}\frac{\eta_j}{u_j}\rightarrow\infty\mbox{ and }\frac{h_{R,l}(k_j, u_j)}{h_{R,r}(\eta_j, u_j)}\ \gtrsim\ \frac{\omega(\eta_j)}{\omega(u_j)}.$$
\end{prop}

\begin{proof}
    Since $\mathcal{B}$ is not conservative, Proposition \ref{p3} gives $$\sup_{(m,u)\in \mathbb{D}}\frac{h_{R,l}(m,u)}{h_{R,r}(m,u)} \ =\ \infty.$$
    Let $(z_n, v_n')\in \mathbb{D}$ be chosen such that 
    \begin{equation}\label{ree20}
        \lim_{n\rightarrow \infty}\frac{h_{R,l}(z_n, v_n')}{h_{R,r}(z_n, v_n')}\ =\ \infty.
    \end{equation}
    Let $\psi$ be a characteristic function of $h_{R,l}$ that is also in the definition of Property (I). By the definition of characteristic functions, we have
    \begin{equation}\label{ree21}h_{R,l}(z_n, v_n')\ \lesssim\ h_{R,l}(z_n, \psi(z_n)).\end{equation}
    We now build a sequence $(v_n)$:
    $$v_n\ =\ \begin{cases}
    \psi(z_n) &\mbox{ if } v_n'>\psi(z_n),\\
    v_n' &\mbox{ if }v_n'\le \psi(z_n).
    \end{cases}$$
    Note that if $v_n'>\psi(z_n)$, we have
    $$\frac{h_{R,l}(z_n, v_n)}{h_{R,r}(z_n, v_n)}\ =\ \frac{h_{R,l}(z_n, \psi(z_n))}{h_{R,r}(z_n, \psi(z_n))}\ \stackrel{\eqref{ree21}}{\gtrsim}\ \frac{h_{R,l}(z_n, v_n')}{h_{R,r}(z_n, v_n')},$$
    which, along with \eqref{ree20}, implies that 
    \begin{equation}\label{ree22}
         \lim_{n\rightarrow \infty}\frac{h_{R,l}(z_n, v_n)}{h_{R,r}(z_n, v_n)}\ =\ \infty\mbox{ and }v_n\ \le\ \psi(z_n). 
    \end{equation}
    Observe that $h_{R,r}(z_n, v_n)\ge K_b^{-1}\inf_n\|e_n\|$ and $h_{R,l}(z_n,v_n)\le z_n^{1/p}\sup_n\|e_n\|$, so we can assume that $z_n, v_n\rightarrow\infty$. 
    Let $r:\mathbb{N}\rightarrow\mathbb{N}$ be such that $2^{r(n)-1}\le z_n\le 2^{r(n)}$. Since $\omega\in \mathbb{W}_d$, Proposition \ref{doubling} gives $I_\omega<\infty$. Choose $\alpha > I_\omega$ and $C_\alpha$ as in Lemma \ref{auxi}. Choose $(k_j, u_j)$ to be some $(z_{n}, v_{n})$ such that
    \begin{equation}\label{ree23}
        \frac{h_{R,l}(k_j,u_j)}{h_{R,r}(k_j, u_j)}\ \ge\ K_bd^{r(j)}C_\alpha z_j^\alpha\ \ge\ K_bd^{r(j)}\frac{\omega(u_j z_j)}{\omega(u_j)},
    \end{equation}
    where $d$ is equal to the sup in item (1) of Property (I). 
    Set $\eta_j = u_jz_j$. We obtain
    \begin{align}\label{ree24}
        h_{R,r}(\eta_j, u_j)&\ =\ h_{R,r}(u_j z_j, u_j) \ \stackrel{\eqref{ree60}}{\le}\ K_b h_{R, r}(2^{r(j)}u_j, u_j)\nonumber\\
        &\ \le\ K_bd^{r(j)}h_{R,r}(u_j,u_j)\ \lesssim\ K_bd^{r(j)}h_{R,r}(k_j,u_j),
    \end{align}
    where the last two inequalities are due to Property (I) and the fact that $u_j\le \psi(k_j)$. Clearly, $\eta_j/u_j\rightarrow\infty$. Furthermore,
    \begin{align*}
        \frac{h_{R,l}(k_j,u_j)}{h_{R,r}(\eta_j, u_j)}\ \stackrel{\eqref{ree24}}{\gtrsim}\ K_b^{-1}d^{-r(j)}\frac{h_{R,l}(k_j,u_j)}{h_{R,r}(k_j, u_j)}\ \stackrel{\eqref{ree23}}{\ge}\ \frac{\omega(\eta_j)}{\omega(u_j)}.
    \end{align*}
\end{proof}

If our basis is quasi-greedy, we have an immediate corollary.

\begin{cor}
Let $w\in \mathbb{W}_d$. If a quasi-greedy Schauder basis $\mathcal{B}$ has property (I) and is not conservative, then there exist $\eta_j\ge u_j\ge 1$, $j=1,2,\ldots$, such that 
$$\lim_{j\rightarrow\infty}\frac{\eta_j}{u_j}\rightarrow\infty\mbox{ and }\frac{h_{R,l}(u_j, u_j)}{h_{R,r}(\eta_j, u_j)}\ \gtrsim\ \frac{\omega(\eta_j)}{\omega(u_j)}.$$
\end{cor}

\begin{proof}
Use Propositions \ref{rpp1} item (2) and \ref{rpp41}.
\end{proof}

\begin{prop}\label{kppg}
Let $\mathcal{B}$ be a Schauder basis with Property (I) and Property (W$^*$) of a $p$-Banach space. Let $\omega\in \mathbb{W}_d$ with $i_\omega > 0$. If $\mathcal{B}$ is not conservative, then $\mathcal{PG}_q^\omega\hookrightarrow \mathcal{G}_q^\omega$ does not hold for any $q\in (0,\infty]$. 
\end{prop}

\begin{proof}We assume that $q\in (0,\infty)$. (The case $q = \infty$ is similar.)

Step 1: set up. By Proposition \ref{rpp41}, there exist $\eta_j\ge u_j\ge k_j\ge 1$, $j=1,2,\ldots$, such that 
\begin{equation}\label{ree40}\lim_{j\rightarrow\infty}\frac{\eta_j}{u_j}\rightarrow\infty\mbox{ and }\frac{h_{R,l}(k_j, u_j)}{h_{R,r}(\eta_j, u_j)}\ \gtrsim\ \frac{\omega(\eta_j)}{\omega(u_j)}.\end{equation}
Assume that $\mathcal{B}$ has $(C_1, C_2)$-Property (W$^*$). 
Choose $|\varepsilon| = 1$, $\Gamma'_{r,j}\subset\mathbb{N}_{\le C_1\eta_j}$ with $|\Gamma'_{r,j}|= \eta_j$ and 
\begin{equation}\label{ree42}
    \|1_{\varepsilon \Gamma'_{r,j}}\|\ \le\ C_2h_l(\eta_j).
\end{equation}
Define $\Gamma_{r,j} = \Gamma'_{r,j}\cap [u_j+1, \infty)$. Then 
\begin{equation}\label{ree43}
    \|1_{\varepsilon \Gamma_{r,j}}\|\ \le\ (K_b+1)\|1_{\varepsilon \Gamma'_{r,j}}\|\ \stackrel{\eqref{ree42}}{\le}\ C_2(K_b+1)h_{R, r}(\eta_j, u_j).
\end{equation}
Choose $|\delta| = 1$, $\Gamma_{l,j}\subset\mathbb{N}_{\le u_j}$ with $|\Gamma_{l, j}| = k_j$ and 
\begin{equation}\label{ree44}
    \|1_{\delta\Gamma_{l,j}}\|\ =\ h_{R,l}(k_j, u_j).
\end{equation}
Set $x_j := 1_{\delta \Gamma_{l, j}}+2\cdot 1_{\varepsilon \Gamma_{r,j}}$. We have
\begin{eqnarray}\|x_j\|^p\ \le\ \|1_{\delta\Gamma_{l,j}}\|^p + 2^p\|1_{\varepsilon \Gamma_{r,j}}\|^p& \stackrel{\eqref{ree43}, \eqref{ree44}}{\lesssim}& (h_{R,l}(k_j, u_j))^p+(h_{R, r}(\eta_j, u_j))^p\nonumber\\
&\stackrel{\eqref{ree40}}{\lesssim}& (h_{R,l}(k_j, u_j))^p,\end{eqnarray}
which gives 
\begin{equation}\label{ree45}
\|x_j\| \ \lesssim\ h_{R,l}(k_j, u_j). 
\end{equation}
For $n\in \{1, \ldots, \eta_j-u_j\}$, a greedy set of order $n$ of $x_j$ is a subset of $\Gamma_{r,j}$. Therefore, 
\begin{equation}\label{ree46}
    \|\gamma_n(x_j)\| \ \ge\ \frac{1}{K_b}\|1_{\delta \Gamma_{l,j}}\|\ =\ \frac{1}{K_b}h_{R,l}(k_j,u_j).
\end{equation}

Step 2: bound $\|x_j\|_{\mathcal{G}_q^\omega}$.
Let $\zeta = \omega^q$ and $\widetilde{\zeta}(m) = \sum_{n=1}^m \zeta(n)/n, m=1, 2, \ldots$. Then
\begin{eqnarray}\label{ree47}
    \|x_j\|_{\mathcal{G}_q^\omega} & \ge& \left(\sum_{n=1}^{\eta_j-u_j}\frac{1}{n}(\omega(n)\|\gamma_n(x_j)\|)^q\right)^{1/q}\nonumber\\
    & \stackrel{\eqref{ree46}}{\gtrsim}& \left(\sum_{n=1}^{\eta_j-u_j}\frac{1}{n}(\omega(n)h_{R,l}(k_j,u_j))^q\right)^{1/q}\nonumber\\
    & =& (\widetilde{\zeta}(\eta_j-u_j))^{1/q} h_{R,l}(k_j,u_j)\nonumber\\
    & \stackrel{\eqref{zeta}}{\gtrsim} & \omega(\eta_j-u_j) h_{R,l}(k_j,u_j).
\end{eqnarray}

Step 3: bound $\|x_j\|_{\mathcal{PG}_q^\omega}$. For $n\le u_j$, 
\begin{equation}\label{ree49}\|\beta_n(x_j)\|\ \le\ (1+K_b)\|x_j\|\ \stackrel{\eqref{ree45}}{\lesssim}\ h_{R,l}(k_j,u_j),\end{equation}
and for $u_j < n \le C_1\eta_j$, 
\begin{equation}\label{ree50}\|\beta_n(x_j)\|\ \le\ (1+K_b)\|1_{\varepsilon \Gamma_{r,j}}\|\ \stackrel{\eqref{ree43}}{\lesssim}\ h_{R, r}(\eta_j, u_j).\end{equation}
We have 
\begin{eqnarray}\label{ree48}
    &\|x_j\|_{\mathcal{PG}_q^\omega}&\nonumber\\
    & \le& \|x_j\| + \left(\sum_{n=1}^{u_j} \frac{(\omega(n)\beta_n(x_j))^q}{n} + \sum_{n=u_j+1}^{\lceil C_1\eta_j\rceil}\frac{(\omega(n)\beta_n(x_j))^q}{n}\right)^{1/q}\nonumber\\
    & \stackrel{\eqref{ree49},\eqref{ree50}}{\lesssim}& \|x_j\| + \left(\sum_{n=1}^{u_j} \frac{(\omega(n)h_{R,l}(k_j,u_j))^q}{n} + \sum_{n=u_j+1}^{\lceil C_1\eta_j\rceil}\frac{(\omega(n)h_{R, r}(\eta_j, u_j))^q}{n}\right)^{1/q}\nonumber\\
    & \stackrel{\eqref{ree45}}{\lesssim}& h_{R,l}(k_j,u_j) + \left((h_{R,l}(k_j,u_j))^q\widetilde{\zeta}(u_j) + (h_{R, r}(\eta_j, u_j))^q\widetilde{\zeta}(\lceil C_1\eta_j\rceil)\right)^{1/q}\nonumber\\
    & \stackrel{\eqref{zeta}, \omega\in \mathbb{W}_d}{\lesssim}& h_{R,l}(k_j,u_j) + \left((\omega(u_j)h_{R,l}(k_j,u_j))^q + (\omega(\eta_j)h_{R, r}(\eta_j, u_j))^q\right)^{1/q}\nonumber\\
    & \stackrel{\eqref{ree40}}{\lesssim}& \omega(u_j)h_{R,l}(k_j,u_j). 
\end{eqnarray}
By \eqref{ree47} and \eqref{ree48}, we obtain
$$\frac{\|x_j\|_{\mathcal{G}_q^\omega}}{\|x_j\|_{\mathcal{PG}_q^\omega}}\ \gtrsim\ \frac{\omega(\eta_j-u_j)}{\omega(u_j)}\rightarrow\infty \mbox{ due to Theorem \ref{better} and  }\frac{\eta_j-u_j}{u_j}\rightarrow\infty.$$
Hence, $\mathcal{PG}_q^\omega\hookrightarrow \mathcal{G}^\omega_q$ does not hold. 
\end{proof}

\begin{proof}[Proof of Theorem \ref{main2}]
Item (1) follows from the definitions of $\mathcal{PG}_q^\omega$, $\mathcal G_q^\omega$, and partial greediness. We prove item (2). Since $\mathcal B$ is quasi-greedy and Schauder, by Theorem \ref{BPG}, it suffices to show that $\mathcal B$ is conservative, which is clearly true due to Proposition \ref{kppg}. 
\end{proof}

\section{Examples}\label{listexa}

All examples we consider are Banach spaces over real scalars having a Schauder basis. 

\subsection{The summing basis of $c_0$}\label{sumbasis}
Let $\mathcal{B} = (e_n)$ be the canonical basis in $c_0$ and consider the collection 
$(x_n) = (\sum_{i=1}^n e_i), n=1,2, \ldots$, which is a conditional Schauder basis of $c_0$. We have 
$$\left\|\sum_{n=1}^\infty a_nx_n\right\|\ =\ \sup_{N\ge 1}\left|\sum_{n=1}^N a_n\right|.$$

\subsubsection{Calculating democracy functions}
It is easy to see that $h_l(N) = 1$ and $h_r(N) = N, \forall N\in \mathbb{N}$. Similarly, $h_{R,l}(m,u)= m, \forall (m,u)\in \mathbb{D}$ and $h_{R,r}(m,u)=1,\forall (m,u)\in \mathbb{N}\times\mathbb{N}$. Therefore, $\mathcal{B}$ is not conservative. 

\subsubsection{Properties}
We verify each desired property below. 
\begin{enumerate}
    \item Property (W): for every $m, n\in \mathbb{N}$ with $m\ge n$, we let $A = \{m+1, m+2, \ldots, m+n\}$. Clearly, 
    $\|1_A\| = n = h_r(n)$.
    \item Property (W$^*$): for every $m\in \mathbb{N}$, let $A = \{1, 2, \ldots, m\}$ and $\varepsilon = ((-1)^n)_{n=1}^m$ to have 
    $\|1_{\varepsilon A}\| = 1 = h_l(m)$.
    \item Property (I) is due to the fact that $h_{R,r}\equiv 1$. 
\end{enumerate}

\subsection{The difference basis in $\ell_1$}\label{diffbasis}
Consider $\mathcal{B} = (e_n)$, the canonical basis in $\ell_1(\mathbb{N})$ and consider the following vectors:
$$x_1 \ =\  e_1, x_n \ =\ e_n-e_{n-1}, n=2,3,\ldots.$$
The collection $(x_n)_{n=1}^\infty$ is a monotone conditional Schauder basis in $\ell_1$. For $(a_n)_{n=1}^N\subset\mathbb{R}$, 
$$\left\|\sum_{n=1}^Na_nx_n\right\|\ =\ \sum_{n=1}^{N-1} |a_n-a_{n+1}| + |a_N|.$$

\subsubsection{Calculating democracy functions}
We have \begin{align*}h_{R,l}(2N,2N) &\ =\ \|(\underbrace{1,1,\ldots, 1}_{\mbox{ length }2N}, \ldots)\|\ =\ 1\mbox{ and }\\
    h_{R,l}(N, 2N) &\ \ge\ \|(\underbrace{0,1,0\ldots, 0, 1}_{\mbox{ length }2N}, 0, \ldots)\| \ =\ 2N.\end{align*}
    Hence, $h_{R,l}(2N, 2N)= o(h_{R,l}(N, 2N))$, which illustrates item (3) in Proposition \ref{rpp1}.

Due to \cite[Lemma 8.1]{BBGHO}, we know that
    $$h_{r}(N) \ =\ 2N\mbox{ and }h_l(N) \ =\ 1.$$

\begin{prop}
For $(x_n)$ as above, we have
$$h_{R,l}(m,u)\ =\ \begin{cases}2m&\mbox{ if } u\ge 2m,\\2u-2m+1&\mbox{ if }m\le u < 2m,\end{cases}$$
and 
$$h_{R,r}(m,u)\ =\ 2.$$
\end{prop}
\begin{proof}
If $u\ge 2m$, we have
$$h_{R,l}(m,u)\ \ge\ \|(\underbrace{0,1,0\ldots, 0, 1}_{\mbox{ length }2m}, 0, \ldots)\|\ =\ 2m.$$
On the other hand, $h_{R,l}(m,u)\le h_{r}(m) = 2m$. Hence, if $u\ge 2m$, $h_{R,l}(m,u) = 2m$. 

If $m\le u<2m$, assume that $u$ is even. The case for odd $u$ is similar. Consider the problem of distributing $m$ $1$'s among the first $u$ spots to maximize the norm. We distribute $u/2$ 1's as follows:
$$\|(\underbrace{0,1,0\ldots, 0, 1}_{\mbox{ length }u}, 0, \ldots)\|\ =\ u.$$
There are $(m-u/2)$ $1$'s remaining to be distributed. Note that putting an $1$ into the first spot reduces the norm by $1$, while putting an $1$ into other spots reduces the norm by $2$. 

Case 1: if $m-u/2 = 1$, we put the only remaining $1$ into the first spot to obtain norm $u-1 = 2u-2m+1$. 

Case 2: if $m-u/2 \ge 2$, we put an $1$ into the first spot and put other $1$'s into other unfilled spots to obtain norm 
$u-1-2(m-u/2-1) = 2u-2m+1$.

In conclusion, $h_{R,l}(m,u) = 2u-2m+1$ if $m\le u < 2m$.

It is easy to see that $h_{R,r}(m,u) = 2$. 
\end{proof}

\begin{cor}
The difference basis $\mathcal{B}$ is not conservative. 
\end{cor}

\subsubsection{Properties}

We verify each desired property : 
\begin{enumerate}
    \item Property (W): for every $m, n\in \mathbb{N}$ with $m\ge n$, we let $A = \{m+1, m+3, \ldots, m+{2n-1}\}$. Clearly, 
    $\|1_A\| = 2n = h_r(n)$.
    \item Property (W$^*$): for every $m\in \mathbb{N}$, let $A = \{1, 2, \ldots, m\}$ to have 
    $\|1_{A}\| = 1 = h_l(m)$.
    \item Property (I) is due to the fact that $h_{R,r}\equiv 2$. 
\end{enumerate}

\subsection{A modification of the Schreier space}\label{Schreiermodified}
We shall use the same example as in  \cite[Proposition 6.10]{BDKOW}, which is a modification of the Schreier space. Let $\mathbb{S}$ be the completion of $c_{00}$ under the following norm:
$$\|(x_1, x_2, x_3, \ldots)\|_{\mathbb{S}}\ =\ \sup_{F\in \mathcal{F}}\sum_{i\in F}|x_i|,$$
where $\mathcal{F} = \{F\subset \mathbb{N}: \sqrt{\min F}\ge |F|\}$. 
It is easy to check that the canonical basis $\mathcal{B}$ is an $1$-unconditional and $1$-conservative monotone Schauder basis of $\mathbb{S}$.

\subsubsection{Calculating democracy functions}
Let $N\in \mathbb{N}$. 
By the definition of $\|\cdot\|_{\mathbb{S}}$, we have
$$h_r(N) \ =\ N\mbox{ and }h_l(N) = \|(\underbrace{1, \ldots, 1}_{N}, 0,\ldots)\|_\mathbb{S}.$$
For $M\in\mathbb{N}$ and $N\in \mathbb{N}\cup\{0\}$, set $x_{N,M} := \sum_{n=N+1}^{N+M}e_n$ and write $x_{N,M} = (x_1, x_2, x_3, \ldots)$.
\begin{prop}\label{rpp30}
We have
$$\left\|x_{N,M}\right\|_{\mathbb{S}}\ \lesssim \ \sqrt{N+M}.$$
\end{prop}
\begin{proof} 
    Let $F\in \mathcal{F}$ with $\min F = N+j_0$ for some $1\le j_0\le M$. 
    
    Case 1: $\min F + \lfloor \sqrt{\min F}\rfloor - 1 \le N+M$; equivalently, $j_0 + \lfloor \sqrt{N+j_0}\rfloor\le M+1$, which implies that \begin{equation}\label{ree30}j_0\ \le\  M+\frac{5}{2}-\sqrt{M+N+9/4}\ =:\ f(M,N).\end{equation}
    If $f(M,N)<1$, then Case 1 cannot happen. If $f(M,N)\ge 1$, then 
        $$\sum_{i\in F}|x_i| \ \le\ \lfloor \sqrt{\min F}\rfloor\ \le\ \sqrt{N+j_0} \ \stackrel{\eqref{ree30}}{\lesssim}\ \sqrt{N+M}.$$
    
    Case 2: $\min F + \lfloor \sqrt{\min F}\rfloor - 1 \ge N+M$; equivalently, $j_0 + \lfloor \sqrt{N+j_0}\rfloor \ge M+1$, which implies that \begin{equation}\label{ree31}j_0\ \ge\  M+\frac{3}{2}-\sqrt{M+N+5/4}\ =:\ g(M,N).\end{equation}
    Hence, 
    \begin{align*}\sum_{i\in F}|x_i| &\ \le\ (N+M)-\min F + 1\ =\ M-j_0 + 1\\
    &\ \stackrel{\eqref{ree31}}{\le}\ M - g(M,N) + 1\ \lesssim\ \sqrt{M+N}.\end{align*}
\end{proof}

\begin{prop}\label{rpp31}
For $M\in \mathbb{N}$ and $N\in\mathbb{N}\cup\{0\}$, 
$$\left\|x_{N,M}\right\|_{\mathbb{S}}\ \gtrsim \ \begin{cases}M&\mbox{ if }N\ge M^2-1,\\ \sqrt{N+M} &\mbox{ if }N\le M^2-1.\end{cases}$$
\end{prop}

\begin{proof}
Let $F = \{N+j_0, \ldots, N+M\}$, where $1\le j_0\le M$ is the smallest such that
$\sqrt{N+j_0}\ge M-j_0+1$, i.e., $F\in\mathcal{F}$. Equivalently, $$j_0\ \ge\ M+3/2 - \sqrt{M+N+5/4}\ =:\ g(M,N).$$

Case 1: If $g(M,N)\le 1$; equivalently, $N\ge M^2-1$, we choose $j_0 = 1$ and obtain 
$\|x_{N,M}\|_{\mathbb{S}}\ge M$.

Case 2: If $g(M,N)\ge 1$; equivalently, $N\le M^2-1$, we choose $j_0 = \lceil g(M,N)\rceil$. Then
$$\|x_{N,M}\|_\mathbb{S}\ \ge\ M-j_0+1\ \ge\ M-g(M,N)\ \gtrsim\ \sqrt{M+N}.$$
\end{proof}

The following corollaries are immediate from Propositions \ref{rpp30} and \ref{rpp31}. 
\begin{cor}
We have $$h_l(N)\ =\ \|(\underbrace{1, \ldots, 1}_{N}, 0,\ldots)\|_\mathbb{S}\ \asymp\ \sqrt{N},$$
and so $\mathcal{B}$ is not democratic. 
\end{cor}

\begin{cor}\label{rcc1}
For $(m,u)\in \mathbb{D}$, we have 
\begin{equation}\label{ree32}h_{R,l}(m,u)\ \lesssim \ \sqrt{u}\mbox{ and }h_{R,l}(m,u)\ \gtrsim\ \begin{cases}m&\mbox{ if }u\ge m^2+m-1\\ \sqrt{u}&\mbox{ if }u\le m^2+m-1.\end{cases}\end{equation}
\end{cor}
\begin{proof}
    Observe that $h_{R,l}(m,u) = \|x_{u-m,m}\|$ and apply Propositions \ref{rpp30} and \ref{rpp31}.
\end{proof}

\begin{cor}For $(m,u)\in \mathbb{N}\times\mathbb{N}$, 
we have
\begin{equation}\label{ree33}h_{R,r}(m,u)\ \lesssim\ \sqrt{u+m}\mbox{ and }h_{R,r}(m,u)\ \gtrsim\ \begin{cases}m&\mbox{ if }u\ge m^2-1\\ \sqrt{u+m}&\mbox{ if }u\le m^2-1.\end{cases}\end{equation}
\end{cor}
\begin{proof}
    Observe that $h_{R,r}(m,u) = \|x_{u, m}\|$ and apply Propositions \ref{rpp30} and \ref{rpp31}. 
\end{proof}

\subsubsection{Properties}

\begin{enumerate}
    \item Property (W): for every $m, n\in \mathbb{N}$ with $m\ge n$, we let $A = \{m^2+1, m^2+2, \ldots, m^2+n\}$. Clearly, 
    $\|1_A\| = n = h_r(n)$.
    \item Property (W$^*$): for every $m\in \mathbb{N}$, let $A = \{1, 2, \ldots, m\}$ to have 
    $\|1_{ A}\| \asymp \sqrt{m} \asymp h_l(m)$.
    \item Property (I): fix $u\in \mathbb{N}$ and $\ell\in \mathbb{N}\cup \{0\}$. We can assume that $u\le 2^{2\ell}u^2-1$ since the only case that the inequality fails is when $u=1$ and $\ell=0$. We have
    $$h_{R,r}(2^{\ell+1}u, u)\ \lesssim\ \sqrt{2^{\ell+1}u + u}\ \le\ \sqrt{2}\sqrt{2^{\ell}u + u}\ \lesssim\ h_{R,r}(2^\ell u,u),$$
    where the last inequality is due to \eqref{ree33} and the fact that $u\le 2^{2\ell} u^2-1$. 
    Next, we choose a suitable characteristic function of $h_{R,l}$. The function $\psi(m) = (m^2-1)\vee 1$ is a candidate because
    $$\sup_{u} h_{R,l}(m,u)\ \le\ m\ \stackrel{\eqref{ree32}}{\lesssim}\ h_{R,l}(m, (m^2-1)\vee 1).$$
    We check that $h_{R,r}(u,u)\lesssim h_{R,r}(m, u)$ whenever $u\le \psi(m)$: for $m\ge 2$, $\psi(m) = m^2-1$, 
    $$h_{R, r}(u,u)\ \stackrel{\eqref{ree33}}{\lesssim}\ \sqrt{u}\ <\ \sqrt{u+m}\ \stackrel{\eqref{ree33}}{\lesssim}\ h_{R,r}(m, u).$$
    Therefore, the canonical basis has Property (I). 
\end{enumerate}

\subsection{An unconditional basis with Property (I) and Property (W$^*$) but is not conservative.}\label{rex1}

We shall consider the example of a basis given in \cite[Subsection 5.4]{BC}. Let $s_n = \sum_{i=1}^n \frac{1}{i}$, and $\Pi$ be the set of all permutations of $\mathbb{N}$. Let $\mathbb{X}$ be the completion of $c_{00}$ with respect to the following norm:
$$\|(x_1, x_2, x_3, \ldots)\|\ =\ \max\left\{\sup_{n}\frac{1}{\sqrt{s_n}}\sup_{\pi\in \Pi} \sum_{i=1}^n \frac{|x_{\pi(i)}|}{i^{1/2}}, \left(\sum_{i}|x_{2i}|^2\right)^{1/2}\right\}.$$
The canonical basis $\mathcal{B}$ is an 1-unconditional and non-conservative Schauder basis. Indeed, letting $A_N = \{2, 4, 6, \ldots, 2N\}$ and $B_N = \{2N+1, 2N+3, \ldots, 4N-1\}$, we have 
$$\frac{\|1_{A_N}\|}{\|1_{B_N}\|}\ \gtrsim\ \frac{\sqrt{N}}{\sqrt{N}/\sqrt{\ln (N+1)}}\ =\ \sqrt{\ln (N+1)} \ \rightarrow \ \infty.$$
Hence, $\mathcal{B}$ is not conservative. 

\subsubsection{Calculating democracy functions}

For $N\in\mathbb{N}$, it is obvious that 
$$h_r(N) \ =\ \sqrt{N}\mbox{ and }h_\ell(N) \ \asymp\ \frac{\sqrt{N}}{\sqrt{\ln (N+1)}}.$$
Similarly, for $(m,u)\in\mathbb{D}$, $$h_{R,l}(m,u) \ \asymp\ \sqrt{m};$$ for $(m,u)\in\mathbb{N}\times\mathbb{N}$, $$h_{R,r}(m,u) \ \asymp\ \frac{\sqrt{m}}{\sqrt{\ln (m+1)}}.$$

\subsubsection{Properties}

We verify each desired property below. 
\begin{enumerate}
    \item Property (W): for every $m, n\in \mathbb{N}$ with $m\ge n$, we let $A = \{j, j+2, \ldots, j+2(n-1)\}$, where $j$ is the smallest even integer greater than $m$. Clearly, 
    $\|1_A\| = \sqrt{n} = h_r(n)$.
    \item Property (W$^*$): for every $n\in \mathbb{N}$, let $A = \{1, 3, \ldots, 2n-1\}$ to have 
    $\|1_A\| \asymp  h_l(n)$.
    \item Property (I):  Pick $u\in\mathbb{N}$ and $\ell\in \mathbb{N}\cup \{0\}$. Then
    $$\frac{h_{R,r}(2^{\ell+1}u, u)}{h_{R,r}(2^\ell u, u)}\ \asymp\ \frac{\sqrt{2^{\ell+1}u}/\sqrt{\ln(2^{\ell+1}u)}}{\sqrt{2^{\ell}u}/\sqrt{\ln(2^{\ell}u)}}\ = O(1).$$
    Furthermore, choose the characteristic function $\psi(m) = m$. For all $u\le m = \psi(m)$, we have
    $$h_{R, r}(u,u) \ \asymp\ \frac{\sqrt{u}}{\sqrt{\ln (u+1)}}\ \le\ \frac{\sqrt{m}}{\sqrt{\ln (m+1)}}\ \asymp\ h_{R, r}(m,u).$$
\end{enumerate}

\begin{cor}
For the basis $\mathcal{B}$ above, we have $$\mathcal{PG}^\omega_{q}\ \not\hookrightarrow\ \mathcal{G}^\omega_q, \forall \omega\in\mathbb{W}_d\mbox{ with }i_\omega>0, \forall q\in (0,\infty].$$
\end{cor}

\section{Annex: Lower and upper regularity properties}\label{LRP-URP}

Duality for greedy-type bases has been studied in \cite{DKKT}. Some of their results depend on two properties: the upper regularty property (URP) and lower regularity property (LRP). The purpose of this section is to show that URP and LRP are equivalent to the upper dilation index of the parameter been smaller than 1 and the lower one being greater than zero, respectively. 

\begin{defi}\normalfont \label{regularity}  A positive sequence $\omega=(\omega(n))_{n=1}^\infty$ has the lower regularity property, denoted by $\omega\in LRP$, if there exist $\alpha>0$ and $C_\alpha>0$ such that 
\begin{equation}\label{ree1}
    \omega(N) \ \ge\ C_\alpha \left(\frac{N}{k}\right)^\alpha \omega(k), \forall N\ge k.
\end{equation}

A positive sequence $\omega$ has the upper regularity property, denoted by $\omega\in URP$, if there exist $\beta < 1$ and $C_\beta\ge 1$ such that 
\begin{equation}\label{ree2}
    \omega(N)\ \le\ C_\beta\left(\frac{N}{k}\right)^\beta\omega(k), \forall N\ge k.
\end{equation}

We write $\omega\in LRP(\alpha)$ and $\omega\in URP(\beta)$ when \eqref{ree1} and \eqref{ree2} hold, respectively.
\end{defi}

\begin{prop} \label{Prop-9-7}
	Let $\omega\in \mathbb{W}$. Then	
	\begin{enumerate}
	\item  $\omega \in LRP(\alpha)$ if and only if there exists $c>0$ such that $$\varphi_\omega(M)\ \ge\ c M^\alpha, \forall M\ge 1.$$
	
	\item $\omega \in URP(\beta)$ if and only if there exists $c>0$ such that
	$$\Phi_\omega(M)\ \le\ c M^\beta, \forall M\ge 1.$$
	\end{enumerate}
\end{prop}

\begin{proof}
We prove (1). (The proof of (2) is similar.)
If $\omega\in LRP(\alpha)$, then 
	$$\omega(Mk)\ \ge\ C_\alpha M^\alpha \omega(k), \forall M, k\in \mathbb{N}.$$
	Therefore, $\varphi_\omega (M) \ge C_\alpha M^{\alpha}$ for all $M$. 
	
	Conversely, suppose  that $ c':=\inf_{M\ge 1} \varphi_\omega(M)/M^\alpha>0$. 
	Let $N > k$ and choose $M\ge 1$ such that $Mk < N \le (M+1)k$. Then
    $$\omega(N) \ \ge\  \omega(Mk) \ \ge\  c' M^\alpha \omega(k) \ \ge\ c'\left(\frac{M+1}{2}\right)^\alpha\omega(k)\ \ge\ c'2^{-\alpha} (N/k)^\alpha \omega(k).$$
	Therefore, $\omega \in LRP(\alpha)$ with constant $c'/2^\alpha.$
\end{proof} 

\begin{prop} \label{lrp} For $\omega\in \mathbb{W}$, the following are equivalent:
\begin{enumerate}
\item $\omega\in LRP$.
\item $\lim_{M\to \infty} \varphi_\omega(M)=\infty$.
\item There exists  $M_0\ge 2$ such that $\varphi_\omega(M_0)>1$.
\end{enumerate}
\end{prop}

\begin{proof} (1) $\Rightarrow$ (2) follows from Proposition \ref{Prop-9-7}, and (2) $\Rightarrow$ (3) is trivial. We show that (3) $\Rightarrow$ (1). Let $\lambda=\varphi_\omega(M_0)>1$ and $\alpha={\ln \lambda}/{\ln M_0}$. Then 
	$$\omega(M_0k)\ \ge\ \lambda \omega(k)\ =\ M_0^\alpha \omega(k), \forall k\in \mathbb N.$$
Hence, if $k<N$, we can find $j\in \mathbb N_0$ such that $M_0^{j}\le N/k< M_0^{j+1}$ to have
	$$\omega(N)\ \ge\ \omega(M_0^j k)\ \ge\  M_0^{j\alpha} \omega(k)\ \ge\  M_0^{-\alpha} \big({N}/{k}\big)^\alpha \omega(k).$$
	Thus, we have shown that $\omega\in LPR(\alpha)$.
\end{proof}

Similarly, one can prove the following. 

\begin{prop} \label{urp} For $\omega\in \mathbb{W}$, the following are equivalent:\
\begin{enumerate}
\item $\omega\in URP$
\item $\omega\in \mathbb{W}_d$ and $\lim_{M\to \infty} \Phi_\omega(M)/M=0$
\item There exists  $M_0\ge 2$ such that $\Phi_\omega(M_0)/M_0<1$.
\end{enumerate}
\end{prop}

We now deduce the main result in this section.

\begin{thm}\label{better} Let $\omega\in \mathbb{W}$. Then 
\begin{enumerate}
\item $\omega\in LRP$ if and only if $i_\omega>0$. Moreover,
	\begin{equation}\label{lrp1}i_\omega=\sup\{\alpha>0: \omega\in LRP(\alpha)\}.\end{equation}
\item $\omega\in URP$ if and only if $I_\omega<1$. Moreover,
	\begin{equation}\label{urp1}I_\omega=\inf\{\beta<1: \omega\in URP(\beta)\}.\end{equation}
\end{enumerate}
\end{thm}
\begin{proof} We prove (1). (The proof of (2) is similar.) If $\omega\in LRP(\alpha)$, then Proposition \ref{Prop-9-7}.i gives $i_\omega\ge \alpha$. Hence, $\omega\in LRP$ implies $i_\omega>0$. Conversely, if $i_\omega>0$, by \eqref{iIeta}, there exists $M_0\ge 2$ such that 
	$\varphi_\omega(M_0)>1$. So $\omega\in LRP$ by Proposition \ref{lrp}. Finally, it remains to prove ``$\le$'' in \eqref{lrp1}. That is, we must show that if $0<\alpha <i_\omega$ then $\omega\in LRP(\alpha)$. Indeed,  by definition of $i_\omega$, there exists $M_\alpha \ge 2$ such that ${\ln (\varphi_\omega(M))}/{\ln M} \ge \alpha$ for all $M \ge M_\alpha$. Then we conclude using  Proposition \ref{Prop-9-7} item (1).
\end{proof}

\begin{rek}\normalfont The sequence $\omega(n)=\sqrt n (\ln(n+1))^\gamma$, $\gamma\in\mathbb{R}$, shows that, in general, it may not hold that $\omega \in LRP( i_\omega)$ or $\omega \in URP(I_\omega)$. Indeed, in this case $i_\omega=I_\omega=1/2$, but
\begin{equation*}
\varphi_\omega(M)\asymp\sqrt{M}/(\ln(M+1))^{\gamma^-}\mbox{ and }
\Phi_\omega(M)\asymp\sqrt{M}(\ln(M+1))^{\gamma^+}.
\end{equation*}
Hence, by Proposition \ref{Prop-9-7}, $\omega\not\in LRP(1/2)$ if $\gamma<0$, and $\omega\not\in URP(1/2)$ if $\gamma>0$.
\end{rek}


\section*{Acknowledgments} The  authors would like to thank  O. Blasco, G. Garrigós, and T. Oikhberg for the permission to write the results of Subsection \ref{weights} since these notes were in a previous version of \cite{BBGHO} that was never published. Also, thanks to Miguel Berasategui for some discussions during the ellaboration of the paper.


\ \\
\end{document}